\documentclass[11pt]{amsart}
\usepackage{geometry,ulem}
\usepackage{amscd,amssymb,verbatim,xcolor,mathrsfs}
\usepackage{longtable, multirow}
\usepackage{soul,cancel}
\usepackage{apptools}
\date{\today}

\newcommand\al{\alpha}

\newcommand\la{{\lambda}}
\newcommand\bb[1]{{\mathbb     #1}}
\newcommand\C[1]{{\mathcal #1 }}
\newcommand\fk[1]{\mathfrak{#1}}
\newcommand\ovl[1]{\overline{#1}}

\newcommand\mat[2]{\begin{pmatrix}#1\\#2\end{pmatrix}}

\newcommand\be[1]{\begin{#1}}
\newcommand\ee[1]{\end{#1}}

\newcommand\disp{\displaystyle}

\newcommand\pmat{\begin{pmatrix}}
\newcommand\epmat{\end{pmatrix}}

\newcommand\CO{{\mathcal O}}

\newcommand\bC{{\mathbb C}}

\newcommand\bN{{\mathbb N}}

\newcommand\bZ{{\mathbb Z}}

\newcommand\ep{{\epsilon}}
\newcommand\sig{{\sigma}}

\newcommand\fg{{\mathfrak g}}
\newcommand\fh{{\mathfrak h}}
\newcommand\frk{{\mathfrak k}}

\newcommand\fm{{\mathfrak m}}
\newcommand\fn{{\mathfrak n}}

\newcommand\fp{{\mathfrak p}}

\newcommand\ft{{\mathfrak t}}

\newcommand\wti[1]{\widetilde{#1} }
\newcommand\wht{\widehat }

\newcommand\ie{{i.e. ~}}

\newtheorem{theorem}{Theorem}[section]

\newtheorem{corollary}[theorem]{Corollary}
\newtheorem{conjecture}[theorem]{Conjecture}
\newtheorem{definition}[theorem]{Definition}
\newtheorem{example}[theorem]{Example}
\newtheorem{lemma}[theorem]{Lemma}
\newtheorem{proposition}[theorem]{Proposition}
\newtheorem{remark}[theorem]{Remark}

\newcommand\Ad{{\operatorname{Ad}}}
\newcommand\ad{{\operatorname{ad}}}
\newcommand\Ann{{\operatorname{Ann}}}

\newcommand\Ind{{\operatorname{Ind}}}

\begin{document}
\title[Dirac series for complex classical Lie groups]{Dirac series for complex classical Lie groups: A multiplicity-one theorem}
\author{Dan Barbasch}
\author{Chao-ping Dong}
\author{Kayue Daniel Wong}

\address[Barbasch]{Department of Mathematics, Cornell University, Ithaca, NY 14853,
U.S.A.}
\email{barbasch@math.cornell.edu}

\address[Dong]{School of Mathematical Sciences, Soochow University, Suzhou 215006,
P.~R.~China}
\email{chaopindong@163.com}

\address[Wong]{School of Science and Engineering, The Chinese University of Hong Kong, Shenzhen,
Guangdong 518172, P. R. China}
\email{kayue.wong@gmail.com}

\begin{abstract}
This paper computes the Dirac cohomology $H_D(\pi)$ of irreducible
unitary Harish-Chandra modules $\pi$ of complex classical groups viewed as
real reductive groups. More precisely, unitary representations with
nonzero Dirac cohomology are shown to be unitarily induced from
unipotent representations. When nonzero, there is a unique, multiplicity free $K-$type in $\pi$
contributing to $H_D(\pi)$. This confirms conjectures formulated by the first named author and Pand\v zi\'c in 2011.

\end{abstract}

\maketitle
\setcounter{tocdepth}{1}

\section{Introduction}\label{sec:intro}
The Dirac operator was first introduced in the representation theory
of real reductive groups by Parthasarathy \cite{P1, P2} and Schmid in
order to give
geometric realization of the discrete series. A byproduct, the Dirac
inequality, has proved very useful to provide necessary conditions for
unitarity. In the case of real rank one groups, the work of \cite{Baldoni}
and \cite{BB}, shows that this necessary condition is also
sufficient. The Dirac inequality plays a crucial role in the
determination of representations with $(\fk g,K)-$cohomology in the work
of \cite{E} and \cite{VZ}  for complex and real groups, subsequently
expanded by \cite{Sa} to find necessary and sufficient conditions for
the unitarity of irreducible representations with regular integral
infinitesimal character.

In order to find sharper estimates for the spectral gap in the case of
locally symmetric spaces, Vogan in \cite{V2}
introduced the notion of {\bf Dirac cohomology} for irreducible
representations. He  formulated a conjecture on its relationship with
the  infinitesimal character of the representation.

\smallskip

We recall the construction of Dirac operator and Dirac cohomology.
Let $G$ be a connected real reductive Lie group. Fix a Cartan
involution $\theta,$ and write $K:=G^{\theta}$ for the  maximal
compact subgroup. Denote by $\fg_0=\frk_0\oplus\fp_0$ the corresponding Cartan
decomposition of the Lie algebra $\fk g_0$, and $\fk g=\fk k+\fk p$
the corresponding decomposition of the complexification. Let $\langle\
,\rangle$ be an invariant nondegenerate form such that $\langle\ ,\
\rangle\mid_{\fk p_0}$ is positive definite, and $\langle\ ,\
\rangle\mid_{\fk k_0}$ is negative definite. Fix  $Z_1,\dots, Z_n$  an
orthonormal basis of $\fp_0$. Let
$U(\fg)$ be the universal enveloping algebra of $\fg$, and let $C(\fp)$
be the Clifford algebra of $\fp$ with respect to $\langle\,
,\,\rangle$. The {\bf Dirac operator}
$D\in U(\fg)\otimes C(\fp)$ is defined as
$$D=\sum_{i=1}^{n}\, Z_i \otimes Z_i.$$
The operator  $D$ does not depend on the choice of the
orthonormal basis $Z_i$ and  is $K$-invariant for the diagonal
action of $K$ induced by the adjoint actions on both factors.

Define  $\Delta: \frk \to U(\fg)\otimes C(\fp)$ by $\Delta(X)=X\otimes
1 + 1\otimes \alpha(X)$, where $\alpha:\frk\to C(\fp)$ is the
composition of $\ad:\fk k\longrightarrow \mathfrak{so}(\fk p)$ with the embedding  $\mathfrak{s}\mathfrak{o}(\fp)\cong\wedge^2(\fp)\hookrightarrow C(\fp)$. Write $\frk_{\Delta}:=\al(\frk)$, and denote by $\Omega_{\fg}$ (resp. $\Omega_{\frk}$)
the Casimir operator of $\fg$ (resp. $\frk$). Let $\Omega_{\frk_{\Delta}}$ be the image of $\Omega_{\frk}$ under $\Delta$. Then (\cite{P1})
\begin{equation}\label{D-square}
D^2=-\Omega_{\fg}\otimes 1 + \Omega_{\frk_{\Delta}} + (\|\rho_c\|^2-\|\rho_{\fg}\|^2) 1\otimes 1,
\end{equation}
where $\rho_{\fg}$ and $\rho_c$ are the corresponding half sums of positive roots of $\fg$ and $\frk$.

Let
$$
\widetilde{K}:=\{ (k,s)\in K\times {\rm Spin}(\fk p_0)\ :\ \Ad (k)=p(s)\},
$$
where $p: \text{Spin}(\fp_0)\rightarrow \text{SO}(\fp_0)$ is the spin double covering map. If $\pi$ is a
($\fg$, $K$)-module, and if $S_G$ denotes a spin module for
$C(\fp)$, then $\pi\otimes S_G$ is a $(U(\fg)\otimes C(\fp),
\widetilde{K})$ module.

The action of $U(\fg)\otimes C(\fp)$ is
the obvious one, and $\widetilde{K}$ acts on both factors; on $\pi$
through $K$ and on $S_G$ through the spin group
$\text{Spin}\,{\fp_0}$.
The Dirac operator acts on $\pi\otimes S_G$. The Dirac
cohomology of $\pi$ is defined as the $\widetilde{K}$-module
\begin{equation}\label{def-Dirac-cohomology}
H_D(\pi)=\text{Ker}\, D/ (\text{Im} \, D \cap \text{Ker} D).
\end{equation}

The following foundational result on Dirac cohomology, conjectured
by Vogan,  was proven by Huang and Pand\v zi\'c in 2002. Let $\fk h$
be a $\theta-$stable Cartan subalgebra with Cartan decomposition $\fk
h=\fk t+\fk a$ and $\fk t$ a Cartan subalgebra of $\fk k.$

\begin{theorem}[\cite{HP1} Theorem 2.3]\label{thm-HP}
Let $\pi$ be an irreducible ($\fg$, $K$)-module.
Assume that the Dirac
cohomology of $\pi$ is nonzero, and that it contains the $\widetilde{K}$-type with highest weight $\gamma\in\ft^{*}\subset\fh^{*}$. Then the infinitesimal character of $\pi$ is conjugate to
$\gamma+\rho_{c}$ under $W(\fg,\fh)$.
\end{theorem}

\subsection{Dirac Series} Denote by $\widehat{G}$ be the set of
equivalence classes of irreducible unitary $(\fg, K)$-modules. If $\pi \in \widehat{G}$,
then $\pi\otimes S_G$ acquires a natural inner
product, and $D$ is self-adjoint. As a result, Dirac
cohomology simplifies to
\begin{equation}\label{Dirac-unitary}
H_D(\pi)=\text{Ker}\, D=\text{Ker}\, D^2.
\end{equation}
For a unitary irreducible representation,  \eqref{D-square} is a
nonnegative scalar on any $\wti K$-type. If $\chi_\pi$ is the
infinitesimal character of $\pi,$ and $\tau$ is the highest weight of a $\widetilde{K}$-type in
$\pi\otimes S_G,$ then
\begin{equation}
  \label{Dirac-inequality}
  ||\chi_\pi||^2\le ||\tau + \rho_c||^2
\end{equation}
This is {\bf Parthasarathy's Dirac
operator inequality}. Moreover, by Theorem 3.5.2 of \cite{HP2}, the
equality holds precisely when $\tau$ is the highest weight of a $\widetilde{K}$-type
in $H_D(\pi)$ (see Section \ref{sec-spin}).

\vspace{3mm}
Let
$\wht{G}^d$ be the representations with nonzero Dirac cohomology. This
subset forms an interesting part of $\wht{G}$. For convenience, we
call these representations \textbf{Dirac series} of $G$ (terminology suggested
by J.-S. Huang).

\medskip
When $G$ is a complex Lie group
viewed as a real Lie group, a necessary condition for $\pi \in \widehat{G}^d$ is that
twice the infinitesimal character $\la$ of $\pi$ must satisfying the regular integral condition \eqref{eq-HPcomplex2} given in Section \ref{sec-spin}.  For this paper
  we adopt the following setting. We focus on the cases when the infinitesimal character
is {\bf regular half-integral} -- to emphasize, $2\la$ satisfies \eqref{eq-HPcomplex2} but $\la$ is
not integral. This is because in the case of $\la$ regular  integral,
these are unitary representations with
nontrivial $(\fk g,K)-$cohomology, and  the results in \cite{E} and
\cite{VZ} imply that any representation in $\widehat{G}^d$ is
unitarily induced from the trivial representation on a Levi
component. This is not the case for half-integral regular parameter.

\smallskip
We begin by determining the representations with half-integral regular parameter
which are unitary and not unitarily induced from any unitary
representation on a proper Levi component. This can be read off from
\cite{B1} and \cite{V1} for the classical groups, i.e. $GL(n,\mathbb{C}),$ $SO(n,\mathbb{C})$ and $Sp(2n,\mathbb{C})$. 
We give a self contained derivation of the unitary dual at half-integral regular infinitesimal character
for these groups, along with a brief discussion on the cases of genuine representations of the $Spin$ groups. 

\smallskip
For $GL(n,\mathbb{C}),$ these representations are just unitary characters.  Yet this is not the case
for the other classical groups. In \cite{B1}, a larger class of
representations is identified which are called the \textit{building blocks} of
the unitary dual in the sense that
\begin{itemize}
\item they are unitary and are not unitarily induced from unitary
  representations on proper Levi components,
\item any other unitary representation is obtained by unitary
  induction and continuous deformations  from unitarily induced
  modules (complementary series)
\end{itemize}
They turn out to have the additional property that the annihilator in the
universal enveloping algebra is maximal.  We call these
\textit{cuspidal unipotent representations}. Following \cite{BV}, we
consider a larger class of representations which we call
\textit{unipotent}. They have properties analogous to the
representations studied in \cite{BV} which are called \textit{special
  unipotent} and have the properties conjectured by Arthur in relation
to the residual spectrum of locally symmetric spaces.

\smallskip
A general discussion of the notion of \textit{unipotent representation} is
beyond the scope of this paper. We have included an explicit list for
the classical groups and a partial discussion
in Appendix A.  It is a paraphrase of \cite{B3} which identifies the representations
as iterated $\Theta$ lifts from one dimensional  representations.


\medskip
The following conjecture on $\widehat{G}^d$ was formulated in \cite{BP}:
\begin{conjecture}[\cite{BP} Conjecture 1.1] \label{conj:2}
Let $G$ be a connected complex simple Lie group and $\pi \in \widehat{G}$ whose
infinitesimal character is regular and half-integral. Then $\pi \in
\widehat{G}^d$ if and only if $\pi$ is parabolically induced from a unipotent
representation with nonzero Dirac cohomology, tensored with a unitary character.
\end{conjecture}

Conjecture \ref{conj:2} generalizes to real reductive Lie groups, where unitary
induction is replaced by the more general cohomological induction in a range
where unitarity is preserved.
In the complex case, Parthasarathy's Dirac inequality \eqref{Dirac-inequality}
implies that all $\pi \in \widehat{G}$ with regular integral infinitesimal character
are unitarily induced from unitary characters of parabolic subgroups, and hence the conjecture
follows immediately.
\smallskip

Here is the list of all nontrivial unipotent representations
with half-integral regular infinitesimal characters for complex classical groups.
In all cases the representations have maximal primitive ideal. The
parameters are explicit, and fit in the parametrization in Appendix A.
Note that the ones in Type B, C and D are not
induced from unitary representations on proper Levi components.

\begin{description}
\item[Type $A_n$] The infinitesimal character satisfies
  \begin{equation}
  \label{eq:A.1}
2\la=(b-1,b-3,\dots ,a,a-1,\dots,-a+1,-a,\dots ,-b+3,-b+1),
\end{equation}
where we assume $b>a$.
The corresponding unipotent representation is spherical of the form
\begin{equation*}
  \pi_u=\Ind_{GL(a)\times GL(b)}^{GL(a+b)}\left(triv\otimes triv\right).
\end{equation*}
It is also the $\Theta-$lift of the trivial representation of
$GL(2b+1)$ to $GL(2a+2b+1)$.

\item[Type $B_n$] The infinitesimal character $\la$ satisfies
  \begin{equation}
  \label{eq:B.1a}
 2\la= (2b-1,2b-3,\dots,2a+3,2a+1,2a,2a-1,\dots,2,1).
\end{equation}
with $b \geq a$. The nilpotent orbit has columns $(2b+1,2a),$ and the
representation is the  $\Theta-$lift of the trivial representation of
$Sp(2a)$ to $SO(2b+2a+1).$

\item[Type $C_n$] The infinitesimal character satisfies
  \begin{equation}
  \label{eq:C.1}
  2\la=(2n-1,2n-3,\dots ,3,1).
\end{equation}
and there are two representations, the components of the
Segal-Shale-Weil representation. The nilpotent orbit has columns
$(2n-1,1)$ and the representations are the $\Theta-$lifts of the two
characters of $O(1)$ to $Sp(2n).$

\item[Type $D_n$] The infinitesimal character satisfies
  \begin{equation}
  \label{eq:D.1}
2\la=(2b-2,2b,\dots,2a+2,2a,2a-1,2a-2,\dots ,1,0)
\end{equation}
with $b\ge a.$ (When $b=a,$ the parameter is $(2a-1,2a-2,\dots
,1,0)$). There are two representations with maximal primitive
ideal. The nilpotent orbit has columns $(2b,2a-1,1)$ and the
representations are $\Theta-$lifts from the Segal-Shale-Weil
representations which in turn are $\Theta-$lifts of the characters of
$O(1)$. This is a case of two  iterations of $\Theta-$lifts from
1-dimensional representations.
\end{description}
 As already mentioned, the unitarily induced representations from the unipotent ones listed
 above are generalizations of the representations with nontrivial
 $(\fk g,K)-$cohomology. As far as locally symmetric spaces and the work of \cite{A},
 it is expected that they would provide new examples of local factors
 of automorphic forms.

We follow the same strategy in the case of the $Spin$ groups. Here are the parameters
of unipotent representations with half-integral regular infinitesimal characters:
\begin{description}
\item[$Spin(2n+1,\bb C)$] Apart from the infinitesimal characters in \eqref{eq:B.1a},
  \begin{equation}
  \label{eq:B.2}
 2\la= (2n-1,2n-3, \dots, 3,1)/2.
\end{equation}
\item[$Spin(2n,\bb C)$] Apart from the infinitesimal characters in
  \eqref{eq:D.1}, there is also
  \begin{equation}
  \label{eq:D.2}
2\la=(2n-1,2n-3, \dots, 3, \pm 1)/2
\end{equation}
\end{description}
Unlike the parameters in \eqref{eq:B.1a} and  \eqref{eq:D.1}, these parameters correspond to
{\it genuine} representations, i.e. they do not factor through $SO(2n+1,\bb C)$ or
$SO(2n,\bb C)$. Moreover, they have maximal primitive ideal, and are unitarily induced from a unitary
  character of a Levi component of type $A_{n-1}$. Note that half-integral
means $2\la$ is integral, not that the coordinates are half-integers.
 Consequently, just like the case of type A, one only needs to consider 
unitary characters for the genuine representations of Spin groups.

\medskip
We are now ready to state the unitarity
results in \cite{V1} and \cite{B1} for complex classical $G$:
\begin{theorem}[Theorem \ref{thm:unitarydual}] \label{thm:3}
Let $G$ be a classical complex Lie group. Any
$\pi \in \widehat{G}$ with regular, half-integral infinitesimal character is of the form
$$\pi := {\rm Ind}_{MN}^G((\mathbb{C}_{\xi} \otimes \pi_u) \otimes {\bf 1}),$$
where $P=MN$ is a parabolic subgroup of $G$ with Levi factor $M$, and
$\mathbb{C}_{\mu}$ is a unitary character on $M$. Moreover,
$\pi_u$ is either the trivial representation, or a unipotent
  representation with infinitesimal character given in \eqref{eq:A.1}
  -- \eqref{eq:D.1}.
\end{theorem}

By the paragraph after Equation \eqref{eq:A.1}, $\pi_u$ is
  induced from the trivial representation in Type A.  Using
  induction in stages, we will assume from now on that $\pi_u = triv$
  for Type A.

\smallskip
A self-contained proof of Theorem \ref{thm:3} for all classical groups is in Sections 3 to 6. The case of Theorem \ref{thm:3} for Spin groups is also discussed in Section \ref{sec:spinb} and \ref{sec:spind}.
When $\pi$ is not unitary, we will specify precisely on which $K-$types the Hermitian form is indefinite.
This will be useful in proving the analogous theorem for exceptional groups of Type $E$.

Using this, we will prove the following:
\begin{theorem} \label{thm:main1}
Conjecture \ref{conj:2} holds for complex connected classical Lie
groups and the Spin groups.
\end{theorem}

\subsection{Spin-lowest $K-$type} Following \cite{D1}, we are interested
in studying \textbf{spin-lowest $K-$type}
(spin-LKT) of an admissible $(\fg,K)-$module. See Definition \ref{def-spinnorm}
for the precise meaning of
spin-lowest $K-$type in the setting of complex Lie groups.
If $\pi \in \widehat{G}^d$, then the spin-lowest $K-$types are precisely
those contributing to $H_D(\pi).$ More explicitly,
let $\tau$ be the highest weight of the $\wti{K}-$type occurring $H_D(\pi)$. Then
$$
[V_{\mathfrak{k}}(\tau)\ :\ H_D(\pi)]=\sum_{\eta\  {\rm spin-LKT}} [V_{\mathfrak{k}}(\eta) : \pi]\cdot [V_{\mathfrak{k}}(\eta) \otimes S_G\ :\ V_{\mathfrak{k}}(\tau)],
$$
where $V_{\mathfrak{a}}(\eta)$ is the irreducible, finite-dimensional $\mathfrak{a}-$module with highest weight $\eta$.
In view of this, the following conjecture, formulated in \cite{BP}, makes $\wht{G}^d$ and $H_D(\pi)$ precise.
\begin{conjecture}[\cite{BP} Conjecture 4.1 and J.-S. Huang] \label{conj:multone}
Let $G$ be a connected complex simple Lie group, and $\pi \in \widehat{G}^d$.
Then $\pi$ has a unique spin-lowest $K-$type $V_{\mathfrak{k}}(\eta)$ which occurs with multiplicity one.
\end{conjecture}

Here is the second main result of this paper:
\begin{theorem} \label{thm:main}
Conjecture \ref{conj:multone} holds for complex connected classical Lie groups and the Spin groups.
\end{theorem}

We believe that Theorems \ref{thm:main1} and \ref{thm:main} should hold 
for all complex reductive groups. 
Indeed, based on the results in \cite{DD, D2, DW}, these theorems are shown to be true 
for exceptional groups of type $G_2$, $F_4$, $E_6$ and $E_7$. We give full details on the case of complex $E_8$ in a forthcoming work.

\bigskip

The manuscript is organized as follows. Section 2 includes some preliminary results on complex simple Lie groups,
Dirac cohomology and spin-lowest $K-$types. Sections 3--6 state the classification of the unitary dual for complex
classical Lie groups with half integral regular infinitesimal
character (cf. \cite{B1}, \cite{V1}) and gives complete proofs.  Section 7 proves a stronger version of Conjecture
\ref{conj:multone} for unipotent representations, which  is essential for the determination of
$H_D(\pi)$ in Section 8. In Appendix A, we give an overview of unipotent representations for complex classical Lie groups.
Finally, in Appendix B, we present some calculations on \texttt{atlas} (\cite{ALTV}, \cite{At}) for the modules appearing in Sections 4--6, offering examples for the results in these sections.

\section{Preliminaries}
Let $G$ be a connected complex simple Lie group viewed as a real Lie group. Fix a
maximal compact subgroup $K$ and a Borel subgroup $B$.
Then $T := K \cap B$ is a maximal torus in $K$.

Denote by $\mathfrak{t}_0$ the Lie algebra of $T$. Then $\mathfrak{a}_0 := \sqrt{-1} \mathfrak{t}_0$ is a maximally
split Cartan subalgebra of $\mathfrak{g}_0$. Let $A := \mathrm{exp}(\mathfrak{a}_0)$. Then
$H = TA$ is a Cartan subgroup of $G$ with Lie algebra $\mathfrak{h}_0=\fk t_0+\fk a_0$.

The realization of the complexification of $\mathfrak{g}_0$ in  (2.1.3) -- (2.1.7) of \cite{B1} gives
\begin{equation}\label{identifications}
\mathfrak{g}\cong \mathfrak{g}_{0} \oplus
\mathfrak{g}_0, \quad
\mathfrak{h}\cong \mathfrak{h}_{0} \oplus
\mathfrak{h}_0,\quad \mathfrak{t} \cong \{(x,-x) : x\in
\mathfrak{h}_{0} \}, \quad \mathfrak{a} \cong\{(x, x) : x\in
\mathfrak{h}_{0} \}
\end{equation}
(we drop the subscripts of the Lie algebras to denote their complexifications).

Let $\rho$ be the half sum of positive roots
in $\Delta_G^{+}$. A choice of positive roots of $\fk g$ is
$$
\Delta^+(\fg, \fh)=\{ \al\times 0\}\cup \{0\times (-\al)\}_{\al\in \Delta_G^+}.
$$

Denote by $W$ the Weyl group $W(\fg_0,
\fh_0)$, which has identity element $e$ and longest element
$w_0$. Then $W(\fg, \fh)\simeq W \times W$.

\subsection{Classification of irreducible modules}
The classification of irreducible $(\fk g,K)$-modules for complex Lie
groups was first obtained by Parthasarathy-Rao-Varadarajan \cite{PRV}
and Zhelobenko \cite{Zh}. Let $(\lambda_{L}, \lambda_{R})\in
\fh_0^{*}\times \fh_0^{*}$ be such that $\lambda_{L}-\lambda_{R}$ is
a weight of a finite dimensional holomorphic representation of $G$.
Using \eqref{identifications}, we can view $(\lambda_L, \lambda_R)$ as a real-linear functional on $\fh$ (we will also sometimes denote it as $\mat{\la_L}{\la_R}$), and write $\bC_{(\lambda_L, \lambda_R)}$ as the character of $H$ with differential $(\lambda_L, \lambda_R)$ (which exists) with
$$
\bC_{(\lambda_L, \lambda_R)}|_{T}= \bC_{\mu} := \bC_{\lambda_L-\lambda_R}, \quad \bC_{(\lambda_L, \lambda_R)}|_{A}= \bC_{\nu} :=\bC_{\lambda_L+\lambda_R}.
$$

Put $X(\lambda_{L}, \lambda_{R})
:=K\mbox{-finite part of Ind}_{B}^{G}
(
\bC_{(\lambda_L, \lambda_R)} \otimes {\bf 1}
)
.$

\begin{theorem}\label{thm-Zh} {\rm (\cite{PRV}, \cite{Zh})}
The $K-$type with extremal weight $\mu:=\lambda_{L}-\lambda_{R}$
occurs with multiplicity one in
$X(\lambda_{L}, \lambda_{R})$. Let
$J(\lambda_L,\lambda_R)$ be the unique subquotient of
$X(\lambda_{L}, \lambda_{R})$ containing this
$K-$type.
\begin{itemize}
\item[a)] Every irreducible admissible ($\fg$, $K$)-module is of the form $J(\lambda_L,\lambda_R)$.
\item[b)] Two such modules $J(\lambda_L,\lambda_R)$ and
$J(\lambda_L^{\prime},\lambda_R^{\prime})$ are equivalent if and
only if there exists $w\in W$ such that
$w\lambda_L=\lambda_L^{\prime}$ and $w\lambda_R=\lambda_R^{\prime}$.
\item[c)] $J(\lambda_L, \lambda_R)$ admits a nondegenerate Hermitian form if and only if there exists
$w\in W$ such that $w(\lambda_L-\lambda_R) =\lambda_L-\lambda_R , w(\lambda_L+\lambda_R) = -\overline{(\lambda_L+\lambda_R)}$.
\end{itemize}
\end{theorem}

The $W\times W-$orbit of $(\lambda_L, \lambda_R)$ is the infinitesimal character of $J(\lambda_L, \lambda_R)$.

In general we normalize hermitian forms on irreducible modules
  to be positive on the lowest $K-$type. Occasionally we will say that
  the form is indefinite on a set of $K-$types, with the understanding
  that if one of them is a lowest $K-$type, then the form is
  normalized as stated above.

\subsection{PRV-component}
In this subsection, we summarize Corollaries 1 and 2 to Theorem 2.1 of \cite{PRV}
on the decomposition of the tensor product $V_{\mathfrak{k}}(\sigma_1)\otimes V_{\mathfrak{k}}(\sigma_2)$ for highest weights
$\sigma_1$ and $\sigma_2$.

\begin{theorem}\label{thm-PRV}{\it (\cite{PRV})}
The component $V_{\mathfrak{k}}(\{\sigma_1+w_0\sigma_2\})$ occurs exactly once in $V_{\mathfrak{k}}(\sigma_1)\otimes V_{\mathfrak{k}}(\sigma_2)$,
where $\{\sigma_1+w_0\sigma_2\}$ is the unique dominant element to which
$\sigma_1+w_0\sigma_2$ is conjugate under the action of $W$.
Moreover, any other component $V_{\mathfrak{k}}(\eta^\prime)$ occurring in $V_{\mathfrak{k}}(\sigma_1)\otimes V_{\mathfrak{k}}(\sigma_2)$ must be of the form
$$
\eta^\prime=\{\sigma_1+w_0\sigma_2\}+\sum_{i=1}^{l} n_i \alpha_i, \mbox{  where } n_i\in\bN.
$$
In particular,
$$
\|\{\sigma_1+w_0\sigma_2\}+\rho\| < \|\eta^\prime+\rho\|.
$$
\end{theorem}

The factor $V_{\mathfrak{k}}(\{\sigma_1+w_0\sigma_2\})$ is usually called the {\bf PRV-component} of $V_{\mathfrak{k}}(\sigma_1)\otimes V_{\mathfrak{k}}(\sigma_2)$.

\subsection{Hermitian modules with Dirac cohomology}\label{sec-spin}
Let $\pi$ be an irreducible $(\fg,K)-$module for a complex Lie group $G$.
By Theorem \ref{thm-HP} and \eqref{identifications}, $\pi$ has Dirac cohomology
if and only if its Zhelobenko parameter $(w_1\la_L,w_2\la_R)$ satisfies
\begin{equation} \label{eq-HPcomplex}
\begin{cases}
  w_1\la_L-w_2\la_R=\tau +\rho\\
  w_1\la_L +w_2\la_R=0,
\end{cases}
\end{equation}
where $V_{\mathfrak{k}}(\tau)$ is a $\widetilde{K}-$type in $H_D(\pi)$. The second equation implies $\la_R=-w_2^{-1}w_1\la_L.$ Since $\tau
+\rho$ is regular integral, the first equation implies that
$2w_1\la_L$ is regular integral.

Write $\la=w_1\la_L.$ The module can be written as $\pi = J(\la,-s\la)$ with $2\la$ regular integral,
and the first equation of \eqref{eq-HPcomplex} implies that the {\it only} $\widetilde{K}-$type
that can appear in $H_D(\pi)$ is $V_{\mathfrak{k}}(2\la - \rho)$.
Furthermore, if $J(\la,-s\la)$ is Hermitian (e.g. if $J(\la,-s\la)$ is unitary),
it follows as in \cite{BP} that $s$ is an involution.

\medskip
Assume further that $\pi = J(\la,-s\la) \in \widehat{G}$, i.e. it is unitary. To relate the above arguments in
terms of Parthasarathy's Dirac inequality, note
that $V_{\mathfrak{k}}(\tau)$ is in $H_D(\pi)$ if and only if
\begin{equation} \label{eq-HPcomplex2}
2\la = \tau + \rho,
\end{equation}
which is precisely when the equality holds in \eqref{Dirac-inequality}. Moreover, if the
$K-$type $V_{\mathfrak{k}}(\eta)$ in $\pi$ contributes to $H_D(\pi)$, then by Theorem \ref{thm-PRV}
it must come from the PRV component of
$$V_{\mathfrak{k}}(\eta) \otimes S_G = 2^{[\frac{l}{2}]}\ V_{\mathfrak{k}}(\eta) \otimes V_{\mathfrak{k}}(\rho),$$
where the equality comes from Lemma 2.2 of \cite{BP}.
This leads to the following
definition given in \cite{D1}.
\begin{definition} \label{def-spinnorm}
The {\bf spin norm} of the $K-$type $V_{\mathfrak{k}}(\eta)$ is defined as
\begin{equation}\label{Spin-norm-K-type}
\|\eta\|_{\mathrm{spin}} := \| \{\eta-\rho\} + \rho \|
\end{equation}
For any irreducible admissible $(\fg, K)-$module $\pi$, we define
\begin{equation}\label{spin-norm-X}
\|\pi\|_{\mathrm{spin}}:=\min \|\eta\|_{\mathrm{spin}},
\end{equation}
where $\eta$ runs over all the $K-$types occurring in $\pi$.
A module  $V_{\mathfrak{k}}(\eta)$ is called a {\bf spin-lowest $K-$type} of $\pi$ if
it occurs in $\pi$ and
$\|\eta\|_{\mathrm{spin}}=\|\pi\|_{\mathrm{spin}}$.
\end{definition}

Using the terminology in Definition \ref{def-spinnorm}, the results of this section can be summarized as follows.

\begin{proposition}\label{prop-D-spin-lowest}
Let $\pi = J(\la,-s\la) \in \widehat{G}$ with $2\la$ regular integral, and $s\in
W$ an involution. Then $\|\pi\|_{spin} \geq \| 2\la \|$, and the equality holds
if and only if $J(\la,-s\la) \in \widehat{G}^d$.

In such cases, $H_D(\pi)$ consists of a single $\widetilde{K}-$type $V_{\mathfrak{k}}(2\lambda - \rho)$ with multiplicity
\begin{align*}
[V_{\mathfrak{k}}(2\lambda - \rho): H_D(\pi)] &= \sum_{\eta\ \text{spin}-LKT} [V_{\mathfrak{k}}(\eta): \pi]\cdot [V_{\mathfrak{k}}(\eta) \otimes S_G: V_{\mathfrak{k}}(2\lambda - \rho)] \\
&= 2^{[\frac{l}{2}]} \sum_{\eta\ \text{spin}-LKT} [V_{\mathfrak{k}}(\eta): \pi]\cdot [V_{\mathfrak{k}}(\eta) \otimes V_{\mathfrak{k}}(\rho): V_{\mathfrak{k}}(2\lambda - \rho)]\\
&= 2^{[\frac{l}{2}]}\sum_{\eta\ \text{spin}-LKT} [V_{\mathfrak{k}}(\eta): \pi].
\end{align*}
\end{proposition}
Conjecture \ref{conj:multone} can be rephrased in the following sharper form. This is the main result of the paper in the case of groups of classical type.
\begin{conjecture}\label{conj:sharp}
Let $\pi = J(\lambda,-s\lambda) \in \widehat{G}$. Then
$$[\pi \otimes V_{\mathfrak{k}}(\rho)\ :\ V_{\mathfrak{k}}(2\lambda - \rho)] := \sum_{\kappa} [V_{\mathfrak{k}}(\kappa)\ :\ \pi]\cdot [V_{\mathfrak{k}}(\kappa) \otimes  V_{\mathfrak{k}}(\rho)\ :\ V_{\mathfrak{k}}(2\lambda - \rho)] = \begin{cases} 1 & \text{if }\pi \in \widehat{G}^d \\ 0 &\text{otherwise} \end{cases}.$$
Consequently, if $\pi \in \widehat{G}^d$, then $H_D(\pi) = 2^{[\frac{l}{2}]}V_{\mathfrak{k}}(2\lambda - \rho)$ by Proposition \ref{prop-D-spin-lowest}.
\end{conjecture}

\section{Unitary Dual} \label{sec-dual}
We use the notation and terminology in the previous section. We determine the unitary representations $J(\lambda, -s\lambda)$
with $2\lambda$ regular and integral; as already mentioned,  $s$ must be an involution. The results were first proved in \cite{B1} and \cite{V1}, and can be summarized as follows.

\begin{theorem}[\cite{B1}, \cite{V1}] \label{thm:unitarydual}
Let $G$ be a classical complex Lie group. Any irreducible unitary representation
$\pi := J(\lambda, -s\lambda)$ of $G$ with $2\lambda$ regular and integral
must be of the form
$$\pi := {\rm Ind}_{LU}^G((\mathbb{C}_{\mu} \otimes \pi_u) \otimes {\bf 1} ),$$
where $P=LU$ is a parabolic subgroup of $G$ with Levi factor $L$, $\mathbb{C}_{\mu}$ is a unitary character of  $L$,
and $\pi_u$ is either the trivial representation, or one of the unipotent representations listed in \eqref{eq:B.1a} -- \eqref{eq:D.1} for Type B, C or D:
\begin{description}
\item[Type $B_n$] The spherical unipotent representations
$$\pi_u =  J\begin{pmatrix}-b+1/2,\dots , -1/2; -a,\dots , -1\\ -b+1/2,\dots , -1/2; -a,\dots , -1\end{pmatrix},
\quad 0 < a \le b \text{ integers and } a+b = n.$$
   It has $K-$spectrum
$$
V_{\mathfrak{k}}(\alpha_1, \alpha_1, \dots, \alpha_a, \alpha_a, \underbrace{0, \dots, 0}_{b-a}),\quad \alpha_1\geq \cdots \geq\alpha_a\geq 0.
$$
\item[Type $C_n$] The Oscillator representations
$$\pi_u^{even} = J\begin{pmatrix}-n+1/2,&\dots ,&-1/2\\-n+1/2,&\dots ,&- 1/2\end{pmatrix}
\text{ and } \pi_u^{odd} = J\begin{pmatrix}-n+1/2,&\dots ,&-1/2\\-n+1/2,&\dots ,& 1/2\end{pmatrix},$$
Their $K-$spectra are given by
$$V_{\mathfrak{k}}(2k,0,\dots ,0)\quad \text{and}\quad V_{\mathfrak{k}}(2k+1,0,\dots, 0), \quad k \geq 0$$
\item[Type $D_n$] The unipotent representations
\begin{align*}
\pi_u^{even} &= J\begin{pmatrix}
-a+1/2,&\dots ,&-3/2,&-1/2,&;-b+1,&\dots ,& 0\\
-a+1/2,&\dots ,&-3/2,&-1/2,&;-b+1,&\dots ,& 0
\end{pmatrix}
  \text{ and }\\
\pi_u^{odd} &= J\begin{pmatrix}
-a+1/2,&\dots ,&-3/2,&-1/2,&;-b+1,&\dots ,& 0\\
-a+1/2,&\dots ,&-3/2,& 1/2,&;-b+1,&\dots ,& 0
\end{pmatrix}
\end{align*}
  with $0 < a\le b$ integers and $a+b = n$. Their $K-$spectra are
$$
V_{\mathfrak{k}}(\alpha_1,  \dots,  \alpha_{2a}, \underbrace{0, \dots, 0}_{b-a}),\quad \alpha_1\geq \cdots \geq \alpha_{2a}\geq 0,\quad \sum_i \al_i \text{ is even/odd.}
$$
\end{description}
\end{theorem}

\subsection{Bottom Layer $\mathbf{K-}$types}\label{sec:bottom}
We use the standard realizations of the classical groups and Lie algebras.
As in \cite{B2}, we will use the notion of relevant $K-$types to detect non-unitarity of
$\pi$.
\begin{definition}
The $K-$types $V_{\mathfrak{k}}(1,\dots ,1,0,\dots ,0,-1,\dots ,-1)$ with equal
number of $1$ and $-1$ for type A, and $V_{\mathfrak{k}}(1,\dots ,1,0,\dots ,0)$ and $V_{\mathfrak{k}}(2,1,\dots
,1,0,\dots ,0)$ in types B, C, D will be called \textbf{cx-relevant}.  The ones with coordinates $\pm 1$ only, will be called \textbf{fundamental cx-relevant}.
\end{definition}

We will make heavy use of bottom layer $K-$types as detailed in
\cite{KnV}. The special case of complex groups  is in Section 2.7 of \cite{B1}.
For the classical groups of Type B, C or D, the results in coordinates are as follows.
Write the lowest $K-$type of $J(\lambda,-s\lambda)$ as
$$
\mu=(\dots ,\underset{\mu_r}{\underbrace{r,\dots ,r}},\dots
,\underset{\mu_1}{\underbrace{1,\dots
    ,1}},\underset{\mu_0}{\underbrace{0,\dots ,0}}) = (\dots,
r^{\mu_r},\dots,1^{\mu_1},0^{\mu_0}).
$$
Let
$$
\begin{aligned}
&{M_1=\prod_{r\geq 1} GL(\mu_r) \times G(\mu_0)} && J_1=\bigotimes_{r\geq 1}
J_{GL(\mu_r)}(\lambda_L^r,\lambda_R^r) \otimes
J_{G(\mu_0)}(\lambda_L^0,\lambda_R^0)\\
&{M_2=\prod_{r \geq 2} GL(\mu_r) \times G(\mu_1 + \mu_0)} &&
J_2=\bigotimes_{r\geq 2} J_{GL(\mu_r)}(\lambda_L^r,\lambda_R^r) \otimes
J_{G(\mu_1 + \mu_0)}(\lambda_L^1 \cup \lambda_L^0,\lambda_R^1 \cup
\lambda_R^0)
\end{aligned}
$$
be Levi components of
real parabolic subalgebras containing the centralizer of $\mu,$ and
irreducible modules. Let
\begin{equation} \label{eq:inclusion}
I_1 := \Ind_{M_1}^G(J_1), \quad I_2:= \Ind_{M_2}^G(J_2)
\end{equation}
be induced modules containing $J(\la,-s\la)$. We only specify the information
on the Levi subgroup for parabolic induction when there is no danger of confusion.
Bottom layer $K-$types
are of the form $\mu_i=\mu+\mu_{M_i}$ where
$\mu_{M_i}$ are $K\cap M_i-$types in $J_i$  so that $\mu_i$
is dominant. They possess the  crucial property that the multiplicities and
signatures of $\mu_{M_i}$ on the $J_i$ and $\mu_i$ in the induced
modules in \eqref{eq:inclusion} and the lowest $K-$type factor $J$
coincide.  By Section 2.7 of \cite{B1}, some of the bottom layer $K-$types for $I_1$ are
obtained by adding $(1,\dots ,1,0,\dots ,0,-1,\dots , -1)$ (equal
number of $1$ and $-1$) to the coordinates equal to $r\ge 1$ in $\mu.$ In
addition one can add $(1,\dots ,1,0,\dots ,0)$ to the
  coordinates of $\mu$ equal to $0$; an even number in cases $C,D$.
For $I_2$, there are extra bottom layer
  $K-$types obtained by replacing the coordinates
  $(1^{\mu_1},0^{\mu_0})$ with
  $(2^{\mu_2'},1^{\mu_1'},0^{\mu_0'})$ which also denote a $K\cap M_2-$type
  coming from $J_{G(\mu_1+\mu_0)}$.

\subsection{Necessary Conditions for Unitarity}
\begin{proposition}
  \label{p:lambdar}
Assume that $\la$ is half-integral regular. The parameter
$(\lambda_L^r,\lambda_R^r)$ in \eqref{eq:inclusion}  for $r\ge 1$
consists of at most two strings,
\[\begin{pmatrix}
A,&\dots,&\frac{r}{2} +1, &\frac{r}{2}, &\frac{r}{2} -1, &\dots, &a\\
-a,&\dots,&-\frac{r}{2} +1, &-\frac{r}{2}, &-\frac{r}{2} -1, &\dots, &-A\\
\end{pmatrix},\]
and/or
\[\begin{pmatrix}
B,&\dots,&\frac{r+1}{2}, &\frac{r-1}{2}, &\dots, &b\\
-b,&\dots,&-\frac{r-1}{2}, &-\frac{r+1}{2}, &\dots, &-B\\
\end{pmatrix}
\]
only ($a+A=r,$ and $B+b=r$).
\end{proposition}
\begin{proof}
The irreducible module
$J_{GL(\mu_r)}(\lambda_L^r,\lambda_R^r)$ in \eqref{eq:inclusion} has
$1-$dimensional lowest $K-$type $V_{\mathfrak{k} \cap \mathfrak{gl}(\mu_r)}(r,\dots,r)$. The condition that $2\la$ be regular
integral implies that $J(\la_L^r,\la_R^r)$ is unitarily induced
irreducible from a finite dimensional $J_e\times J_o$ of a Levi component $GL_{e}\times GL_{o}\subset GL(\mu_r)$, where the parameters of $J_e$ and $J_o$ come from the $\mathbb{Z}$ and $\mathbb{Z} + \frac{1}{2}$ coordinates of $J_{GL(\mu_r)}(\lambda_L^r,\lambda_R^r)$
respectively.

Note that by Theorem \ref{thm-Zh}(c), and the assumption that $J(\la_L,\la_R)$ has an invariant Hermitian form, both $J_e$ and $J_o$ have invariant Hermitian forms. Using Casimir's inequality \cite[Lemma 12.6]{V1}, unless $J_e$ and $J_o$ are unitary characters, otherwise $J_{GL(\mu_r)}(\lambda_L^r,\lambda_R^r)$ have indefinite form on $K-$types $V_{\mathfrak{k} \cap \mathfrak{gl}(\mu_r)}(r+1,r,\dots
,r,r-1)$ and $V_{\mathfrak{k} \cap \mathfrak{gl}(\mu_r)}(r,\dots ,r)$. Since these $K-$types are bottom layer in the induced modules \eqref{eq:inclusion}, $J$ is unitary only
if $J_{GL(\mu_r)}(\lambda_L^r,\lambda_R^r)$ is unitary and induced from unitary characters. So   $\mat{\lambda_L^r}{\lambda_R^r}$ must
consist of at most two strings as in the statement of the Proposition.
\end{proof}

\begin{remark} \label{rmk:unitarya}
Since all Levi subgroups of $G = GL(n, \bb C)$ consist only of $GL-$factors, one can apply the above Proposition for all $r \in \mathbb{Z}$ to conclude
that Theorem \ref{thm:unitarydual} holds for Type A. Hence we focus on the classical groups of Type B, C and D from now on.
\end{remark}

\begin{corollary}
  \label{l:regularintegral}

{Assume   $\mathbf{\mu_1 \neq 0}$}. Then
$$
  \mat{\lambda_L^1}{\lambda_R^1}= \begin{cases} {\mat{\frac{1}{2}}{-\frac{1}{2}}}\quad
  \text{ in types } B,C\\
\mat{\frac{1}{2}}{-\frac{1}{2}}\ \text{ or } \
\mat{1,0}{0,-1}  \quad\text{ in type } D.
\end{cases}
$$
\end{corollary}
\begin{proof}
The statement is a direct consequence of the fact that $2\la$ is assumed regular integral.
\end{proof}

We consider
 $J_{G(\mu_1 + \mu_0)}(\lambda_L^1,\lambda_L^0,\lambda_R^1,\lambda_R^0)$ appearing in $J_2$ of \eqref{eq:inclusion}.
A consequence of Proposition \ref{p:lambdar} and Corollary
\ref{l:regularintegral} is that we can write the
parameter as
\begin{equation} \label{eq:larel}
(\la_{rel},-s_{rel}\la_{rel}) := (\lambda^1,\lambda^0,-\lambda^1,
\lambda^0)\quad \text{ with }\quad
\la^1=(\underset{\mu_1}{\underbrace{1,\dots ,1}})\quad  \mu_1=0, 1, 2.
\end{equation}
Specifically, $\la_{rel}=(\la^1,\la^0)$ and $s_{rel}$ is an
involution so that $s_{rel}(\la^1,\la^0)=(-\la^1,\la^0).$
Sections 4--6 is devoted to proving the following:
\begin{theorem} \label{thm:j2} Assume that the parameter is half-integral regular, and $\mu_r=0$ for $r\ge 2$ so that $\la=\la_{rel}.$
Then $J(\la,-s_{rel}\la)$ is unitary if and only if it
is of the form given in Theorem \ref{thm:unitarydual}; i.e. unipotent
tensored with a unitary character. When it is not unitary, the form is indefinite
on cx-relevant $K-$types.
\end{theorem}

\begin{corollary} Let $J(\la,-s\la)$ be an irreducible module with half integral regular
infinitesimal character. Then Theorem \ref{thm:unitarydual} holds.
\end{corollary}
\begin{proof} The Corollary (and therefore Theorem
  \ref{thm:unitarydual}) follows immediately from properties of bottom
  layer $K-$types.
Suppose $J(\la_{rel},-s_{rel}\la_{rel})$ is not of the form given in
Theorem \ref{thm:unitarydual}. Then by Proposition \ref{thm:j2}
it must be non-unitary, which has indefinite form on cx-relevant $K-$types.
Since all cx-relevant $K-$types are bottom layer in $I_2$, this implies that
$J(\la,-s\la)$ is not unitary.

On the other hand, if $J(\la_{rel},-s_{rel}\la_{rel})$ is of the form given in Theorem \ref{thm:unitarydual}, then
by induction in stages $I_2$ is of the form given by Theorem \ref{thm:unitarydual},
with $J(\la,-s\la)$ being its lowest $K-$type subquotient. Since it is
a subquotient of the unitary module $I_2,$ $J$ is unitary. A sharper
result holds -- by Theorem 14.1 of \cite{B1}, $I_2 = J(\la,-s\la)$.
\end{proof}
\subsection{General Strategy} \label{subsec-strategy}
By the corollary above, it suffices to prove Theorem \ref{thm:j2}. In particular, when the parameter
is not as in Theorem \ref{thm:unitarydual},  the form is indefinite on a cx-relevant
$K-$type. These give rise to bottom layer $K-$type in the general case.

\medskip
To treat the case $J((\la^1,\la^0),(-\la^1,\la^0))$ given in Theorem \ref{thm:j2},
the spherical case $J(\la^0,\la^0)$ plays an important role.
Write $\la=\la^0$ from now on.
We define a parabolic subgroup $P(\la)$  and a representation
$\pi_{L(\la)}$ on its Levi component so that
the induced module $I_{P(\la)} := \Ind_{P(\la)}^G(\pi_{L(\la)})$ is Hermitian, and the cx-relevant
$K-$types occur with full multiplicity in the spherical subquotient $J(\la,\la)$.
The induction step proceeds as follows.  Deform $\la$ and the induced module $I_{P(\la)}$
to $\la+t\nu$ where $\nu$ is central for $L(\la)$, so that the norm of the parameter
becomes larger, and the multiplicities of the cx-relevant $K-$types do
not change for small $t.$ Let $t_0>0$ be the nearest where the
multiplicities change; $P(\la+t_0\nu)$ changes as well. If the
condition in Theorem \ref{thm:unitarydual} are not satisfied, the
induction hypothesis holds, so the form is {\it indefinite on cx-relevant
$K-$types}, that is, the form has different signatures
on the lowest $K-$type and at least one of the cx-relevant $K-$types,
and the semi-continuity of the signature implies that the
form was indefinite on cx-relevant $K-$types at $t=0.$ The exceptions
are when $J(\la+t_0\nu,\la+t_0\nu)$ is unitary, or the deformation goes on
to ``$\infty$''. In the first case we find a non-spherical factor in the deformed
induced module with a pair of
indefinite cx-relevant $K-$types. In the second case, the Casimir
inequality implies that the form is indefinite on the trivial and
adjoint $K-$types.

\bigskip
\textit{We will henceforth concentrate on the cases when $\la$ is NOT regular
  integral. The cases when $\la$ is regular integral, are covered by \cite{E};
  the unipotent representations occurring are $\pi_u=Triv.$}

\section{Proof of Theorem \ref{thm:j2} -- Type B}
Let $G = SO(2m+1,\mathbb{C})$ and $K = SO(2m+1)$. The $K-$types have highest weights
$\eta$ with coordinates integers only. Since $\rho=(m-1/2,\dots
,1/2)$,  $2\la=\{\eta - \rho\} + \rho$, $2\la$ must have  integer coordinates only; so $\la$ has integer and half-integer coordinates.
Since we assume that $\la$ is regular half-integral but not integral,
the integral system determined by $\la$ is type $C\times C.$

\subsection{Spherical Representations} \label{sec:4.1}
In the next few subsections, we will prove the following Proposition.
\begin{proposition} \label{thm:unitb}
Let $\la$ be regular half-integral. The spherical irreducible module $J(\la,\la)$ is unitary if and only if it is unipotent, i.e. the parameter is
$$
\la=\left(-K_0+\frac12,\dots,-\frac12; -N_0,\dots, -1 \right)
$$
with $N_0 \leq K_0$. This is a unipotent representation attached to the
nilpotent orbit $[2^{2N_0}1^{2K_0-2N_0+1}]$.

When not unitary, the form is indefinite on the set of cx-relevant $K-$types with highest weights
$$
CXB:=\{(0,\dots ,0),\quad (1,\dots ,1,0,\dots ,0),\quad (2,0,\dots ,0)\}.
$$
\end{proposition}
The unipotent representation in Proposition \ref{thm:unitb} is unitary because it
can be realized via  the dual pair correspondence as
$\Theta(\mathrm{triv}_{Sp})$, from the pair $Sp(2N_0,\bC) \times SO(2K_0+2N_0+1,\bC)$
in the stable range.

\medskip
In order to prove the non-unitarity of other parameters, we use the strategy in Section \ref{subsec-strategy}. We construct an induced module
$I_{P(\la)}$ having  $J(\la,\la)$ as a quotient.
Let $\la$ be half-integral and dominant for the standard
positive system, i.e.
$$\la = (\dots \la_i\ge \la_{i+1} \ge \dots \ge 0),\quad 2\la_i \in \mathbb{N}.$$
If $\la$ is further assumed to be regular, then the above inequalities are strict.
We construct a parabolic subgroup $P(\la) = L(\la)U(\la)$ and an induced module
$I_{P(\la)}$ so that $J(\la,\la)$ is the spherical irreducible
  factor in $I_{P(\la)}$, and the multiplicities of the cx-relevant
  $K-$types are the same.

\begin{itemize}
\item[(i)]
If $1/2$ is a coordinate of $\la,$ form the longest
\textit{string}
\begin{equation*}
  \kappa_0 := (-K_0 + 1/2,\dots ,-1/2)
\end{equation*}
such that all the half-integers starting from $1/2$ to $K_0-1/2$ are coordinates of
$\la$, but $K_0+1/2$ is not. If the coordinate $1$ occurs, form the longest
string
$$\sigma_0 := (-N_0,\dots ,-1)$$
where $N_0$ is the largest integer coordinate that occurs in $\la,$ but $N_0+1$ does not.
Add a factor to $L(\la)$ of type $G(K_0+N_0)=SO(2K_0+2N_0+1)$ and the spherical irreducible
representation with parameter
$$
\begin{pmatrix}
  -K_0+1/2&,\dots ,&-1/2&;&-N_0&,\dots,&-1\\
  -K_0+1/2&,\dots ,&-1/2&;&-N_0&,\dots,&-1
\end{pmatrix}
$$
If $1/2$ is not a coordinate, let $k_1-1/2>0$ be the smallest half-integer coordinate, and form the string $\kappa_1=(k_1-1/2,\dots ,K_1-1/2)$ with increasing coordinates differing by 1 as before. Add a
factor  $GL(K_1-k_1+1),$ and the 1-dimensional representation of $GL(K_1-k_1+1)$ with parameter
$$
\begin{pmatrix}
  k_1-1/2&,\dots ,&K_1-1/2 \\
  k_1-1/2&,\dots ,&K_1-1/2
\end{pmatrix}
$$
to $L(\la)$. Similarly if $1$ does not occur as a coordinate, form $\sig_1=(n_1,\dots ,N_1)$, add a factor $GL(N_1-n_1+1)$ to the Levi component $L(\la)$, and the 1-dimensional representation of $GL(N_1-n_1+1)$ with parameter
$$
\begin{pmatrix}
  n_1 &,\dots ,&N_1 \\
  n_1&,\dots ,&N_1
\end{pmatrix}
$$
\item[(ii)] Remove the coordinates in Step (i) from $\la$, and repeat
  on the remainder until there
  are no half-integer coordinates left. Since the assumption was that at most one
  coordinate was equal to $1/2,$ only $GL-$factors are created.
\item[(iii)] Repeat Steps (i) and (ii) on the integer coordinates
  until there are none left.
\end{itemize}
The process produces a parabolic subgroup, and an induced module on its Levi component. The Levi
component is
\begin{equation} \label{eq:deformb}
L(\la) := \prod_{i > 0} GL(\sigma_j) \times \prod_{j>0} GL(\kappa_i)\times G(K_0+N_0).
\end{equation}
If $\lambda$ is assumed to be regular, its corresponding strings $\kappa_i$, $\sigma_j$
satisfy
\begin{equation}
  \label{eq:medium}
\begin{cases} k_i>2 &\text{if }\ 1/2 \text{ is a coordinate},\\ k_i\ge 2 &\text{otherwise}\end{cases}
\quad \text{and} \quad \begin{cases} n_j> 2 &\text{if }\ 1 \text{ is a coordinate},\\ n_j\ge 2 &\text{otherwise}\end{cases}
\end{equation}

In the proof of Proposition \ref{thm:unitb} below, we begin with $J(\la,\la)$ where $\la$ is regular and
half-integral. Then we deform some $GL$--strings $\kappa_i$, $\sigma_j$, $i, j > 0$ upward and analyze
the new parameter $\la_{new}$ and its corresponding induced module $I(\la_{new})$.
Here $\la_{new}$ is half-integral but is not necessarily regular
(see Example \ref{eg:deforms3} below). Nevertheless, by the above construction of
$\kappa$ and $\sigma-$strings, it is easy to see that the more general parameters satisfy
\begin{equation}
  \label{eq:extraB}
\begin{cases} k_{i+1}-K_i\ge 2,\ \text{or}\\
k_i\le k_{i+1}\le K_{i+1}\le K_i,\end{cases} \quad \text{and} \quad
\begin{cases} n_{j+1}-N_j\ge 2,\ \text{or} \\
n_{i}\le n_{i+1}\le N_{i+1}\le N_{i} \end{cases}.
\end{equation}
We say the strings $\kappa_i$, $\kappa_{i+1}$ (or $\sigma_j$, $\sigma_{j+1}$)
{\bf nested} if its parameters satisfy \eqref{eq:extraB} for all $i, j \geq 0$. The parabolic subgroup is determined by the order of
the factors, and the integer and half-integer strings are interchangeable.
\medskip

The main property of the cx-relevant $K-$types is the following Lemma.

\begin{lemma} \label{lem:bij} Let $\la$ be dominant whose coordinates are
half-integers. Assume that the strings of $\la$ satisfy \eqref{eq:medium} and \eqref{eq:extraB}.
The multiplicities of the cx-relevant $K-$types in $I_{P(\la)}$
coincide with those in $J(\la,\la).$
\end{lemma}
{\begin{proof} This kind of result can be found in
    \cite{B2}. The main difference is that $(2,0,\dots
    ,0)$ is not petite/single petaled. The condition that the value of
    $\check\al$ for $\al$ a long root on the highest weight of the $K-$type be $\le 3$ is
    satisfied except for the case of $(2,0,\dots ,0)$ and a long root.
    The crucial property needed is that
    $SL(2)-$intertwining operators be isomorphisms on these $K-$types. Condition
    \eqref{eq:medium} insures that this property is still valid for
    the larger class of $K-$types. We sketch the steps.

Recall that $\la$ was assumed dominant. Then $J(\la,\la)$ is the image
of the long intertwining operator from $I_B(\la,\la)$ to
$I_B(-\la,-\la).$ The module $I_{P(\la)}$ is a homomorphic image of
$I_B(\la,\la).$ The long intertwining operator $A_{w_0}$ factors into
$$
I_B(\la,\la)\longrightarrow I_{P(\la)}\longrightarrow I_B(-\la,-\la).
$$
We only need to show that the intertwining operator on the right is an
injection  on the cx-relevant $K-$types. We need to ``\textit{flip}''
the coordinates of the $\kappa_i$ and $\sig_j$ into their negatives.   This is done by embedding into a larger induced module where it is possible to factor the operator further into ones  induced from $SL(2)_\alpha'$s. Condition
\eqref{eq:medium} insures that they are isomorphisms on the
restrictions of the cx-relevant $K-$types.  This is also the reason
that we have put $\kappa_0$ and $\sig_0$ into the Levi component.
\end{proof}}

We finish this subsection by giving a necessary condition on the spherical parameter:
\begin{lemma}\label{l:bottomb}
If $J(\lambda,\lambda)$ is unitary, then the string $\kappa_0 = (-K_0+1/2, \dots, 1/2)$ must appear in $(\lambda,\lambda)$.
\end{lemma}
\begin{proof}
  The coordinates on the spherical part of $I_1$ in Equation \eqref{eq:inclusion}
	are all $\geq 1.$ The Casimir inequality implies that the form is indefinite on the adjoint $V_{\mathfrak{k}}(1,1,0,\dots ,0)$ $K-$type and the trivial $K-$type $V_{\mathfrak{k}}(0,\dots,0)$ . These give rise to bottom layer $K-$types of $I_1$,
	and hence the irreducible $J(\lambda,\lambda)$ is not unitary.
\end{proof}

\subsection{Proof of Proposition \ref{thm:unitb} -- $\mathbf{\la=\kappa_0\cup\sig_0}$} \label{sec-bn>k}
If only $\sig_0$ occurs in $\la$, then it is not unitary by Lemma \ref{l:bottomb}.
Furthermore, the case when $\la = \kappa_0\cup\sig_0$ with $K_0\ge N_0$ is unitary. So assume
\begin{equation}
  \label{zerostrings}
 \la = \kappa_0\cup \sigma_0\quad \text{  satisfying }\quad  N_0>K_0\ge 1.
\end{equation}
Let
$$
Ind(\la_t):=\Ind^G_{GL(\sigma_0)\times G(K_0)}(\left(1+t,\dots, N_0+t\right) \otimes \mathrm{triv}),
$$
The signatures and multiplicities of the fundamental cx-relevant $K-$types of the form
$V_{\mathfrak{k}}(1,\dots ,1,0,\dots ,0)$ coincide on $Ind(\la_0)$ and $J(\la,\la)$.
Indeed, $Ind(\la)$ is a homomorphic image of $\Ind_B^G(\la,\la)$, and
the intertwining operator changing $(1,\dots ,N_0)$ to
  $(-N_0,\dots ,-1)$ involves only $(\check\al,w\la)$ which are integers $\ge 2$:
  $$
    \mat{i}{i}\mapsto \mat{-i}{i}
  $$
  The kernel of the intertwining operator has lowest $K-$type of highest weight $(2i)$ for $1 \leq i \leq N_0$.
  So the intertwining operator is an isomorphism on the cx-relevant $K-$types $V_{\mathfrak{k}}(1,\dots ,1,0,\dots ,0)$
(but {\it not} necessarily for $V_{\mathfrak{k}}(2,0\dots ,0)$).
These values remain unchanged for all
$Ind(\la_t)$ with $t \in [0,1/2)$ because the multiplicities do not change. At $t=1/2,$
\begin{align*}
\la_{1/2} &= (3/2, 5/2, \dots, N_0 + 1/2; -K_0+1/2, \dots, -1/2)\\
&= (-N_0-1/2,\dots ,-1/2) \cup (3/2,\dots ,K_0-1/2).
\end{align*}
So the induced module $I_{P(\la_{1/2})}$ defined in Section \ref{sec:4.1} is given by
$$
I_{P(\la_{1/2})} = \Ind^G_{GL(K_0 - 1)\times G(N_0)}((3/2,\dots ,K_0-1/2) \otimes \mathrm{triv}),
$$
and differs from  $Ind(\la_{1/2})$.
More precisely, apart from $J(\la_{1/2},\la_{1/2})$, $Ind(\la_{1/2})$ has a non-spherical irreducible
factor whose parameter is given by
$$
\begin{pmatrix}
  1/2,\dots ,K_0 -1/2;& 3/2,\dots ,K_0 + 1/2;&K_0+3/2,\dots, N_0+1/2\\
  3/2,\dots, K_0 + 1/2; & 1/2,\dots ,K_0-1/2;&K_0+3/2,\dots ,N_0+1/2
\end{pmatrix}
$$
This module has indefinite form on the $K-$types $V_{\mathfrak{k}}(\underset{2K_0}{\underbrace{1,\dots ,1}},0,\dots ,0)$ and
\begin{equation} \label{eq:brelevant}
\begin{cases}
V_{\mathfrak{k}}(\underset{ 2K_0}{\underbrace{1,\dots ,1}},1) & \text{if}\ N_0 = K_0 + 1;\\
  V_{\mathfrak{k}}(\underset{ 2K_0}{\underbrace{1,\dots ,1}},1,1,0,\dots ,0) & \text{otherwise}
	\end{cases}
	\end{equation}
Indeed, the second $K-$type is bottom layer for the parabolic subgroup with Levi
component $GL(K_0) \times G(N_0-K_0)$. The \textit{spherical part} of
the parameter $
\begin{pmatrix}
  K_0+3/2,\dots, N_0+1/2\\
  K_0+3/2,\dots ,N_0+1/2
\end{pmatrix}
$ is a finite dimensional representation of $G(N_0-K_0)$, so the form is
indefinite on the trivial and adjoint $K-$types of $G(N_0-K_0)$.

Consequently, by semicontinuity of
signatures, $Ind(\la_0)$ and $J(\la,\la)$ also have indefinite form
on the $K-$types given in \eqref{eq:brelevant}.

\subsection{Proof of Proposition \ref{thm:unitb} -- Other Strings}\label{sec:additionalb} Assume  $\la$ contains strings other
than $\kappa_0$ and $\sigma_0$. We do an induction upward on the length of
the parameter, downward on the number of strings.

Assume  there is a $\kappa_i = (k_i-1/2,\dots ,K_i-1/2)$ with $i>0$ or $\sig_j = (n_j,\dots ,N_j)$ with $j>0$.  Replace it by $(k_i-1/2+t,\dots ,K_i-1/2+t)$ (or $(n_j+t,\dots ,N_j+t)$),
  and denote the new parameter by $\la_t.$  At $t=0$, $I_{P(\la)}=I_{P(\la_{0})},$
  and the signatures of cx-relevant $K-$types do not change for $0\le
  t<1/2.$ At $t=1/2,$ if the induction hypothesis (condition for the
  form to be indefinite on the cx-relevant $K-$types) holds for
  $J(\la_{1/2},\la_{1/2})$ we conclude that $J(\la,\la)$ is not
  unitary, with form indefinite on the cx-relevant $K-$types. It may
  happen that $I_{P(\la_{1/2})}$ is unchanged, and we can continue to
  deform $t$ upward. $I_{P(\la)}$ may be  unchanged as
  $t\longrightarrow\infty.$ In this case the form is indefinite on
  the adjoint $K-$type $V_{\mathfrak{k}}(1,1,0,\dots ,0)$. We call this an initial
  case. The other case is when the spherical module
  $J(\la_{1/2},\la_{1/2})$ is unitary. This is the case
  $\sig_0\cup\kappa_0$ with $K_0\ge N_0.$ Note that it includes the
  case when the spherical module is the trivial representation.

  In summary, these cases, which we call \textit{initial cases} are
 \begin{itemize}
\item[(a)] There is a string $\kappa_i$ or $\sig_j$ with $i,j>0$ such that
  $P(\la_t)$ does not change as $t\to\infty$,
\item[(b)] The strings are
  $$
  (-K_0 +1/2, \dots, -1/2; -N_0, \dots, -1),\quad \text{with}\ K_0<N_0$$
  as in the previous section.
\item[(c)] The strings are
$(-K_0 +1/2, \dots, -1/2; -N_0, \dots, -1) \cup \xi$ satisfying
$$
\xi = (K_0,\dots , K_1) \text{ or } (N_0+1/2,\dots , N_1-1/2),
$$
\end{itemize}
so that the deformation of $\xi$ to $t=1/2$ yields a unitary spherical
module. This means that $K_1\geq N_0$ in one case, $K_0\ge N_1$ in  the
other case.
See Example \ref{eg:deforms3} for more
details. In Case (a), as already mentioned, the Casimir inequality implies that the spherical
irreducible module at $t=1/2$ has indefinite form on the trivial and
adjoint $K-$types $V_{\mathfrak{k}}(0,\dots ,0)$ and $V_{\mathfrak{k}}(1,1,0,\dots ,0).$

Case (b) was discussed in the previous section.

For Case (c), we give details for  $\xi =(K_0)$. The other $\xi$ are
similar.  $I_{P(\la_{1/2})}$ has another irreducible factor with
parameter containing
  $$
  \begin{pmatrix}
   - K_0+1/2&-K_0-1/2& K_0 - 3/2& \dots&1/2\\
  -K_0-1/2 &-K_0+1/2& K_0 - 3/2&\dots& 1/2
  \end{pmatrix}
  $$
with the rest of the  spherical part formed of  integer coordinates coming from $\sig_0.$

The lowest $K-$type is $V_{\mathfrak{k}}(1,1,0,\dots ,0)$ and
$V_{\mathfrak{k}}(2,0,\dots ,0)$ is bottom layer. Since for such a parameter the form
on the $GL(2)-$factor is indefinite on $(1,1)$ and
$(2,0)=(1,1)+(1,-1)$, semicontinuity of the signature implies the same
for the parameter at $\la.$

\medskip
The proof of Proposition \ref{thm:unitb} is now complete. \qed

\begin{example} \label{eg:deforms3}
Let $\la = (-11/2,-9/2,-7/2,-5/2,-3/2,-1/2;-1) \cup (3,4) \cup (6)$.
Note that $\kappa_0$ is longer than $\sigma_0$. Deform all $\sigma_i$ into $\kappa_i$ for $i > 0$:
\begin{align*}
\la = &\ (-11/2,-9/2,-7/2,-5/2,-3/2,-1/2;-1) \cup {\bf (3,4) \cup (6)}\\
\longrightarrow &\ (-11/2,-9/2,-7/2,-5/2,-3/2,-1/2;-1) \cup (7/2,9/2) \cup (13/2)\\
= &\ (-13/2,-11/2,-9/2,-7/2,-5/2,-3/2,-1/2;-1) \cup (7/2,9/2)
\end{align*}
Deform the new $\kappa_i$ for $i > 0$ and get
\begin{align*}
&\ (-13/2,-11/2,-9/2,-7/2,-5/2,-3/2,-1/2;-1) \cup {\bf (7/2,9/2)}\\
\longrightarrow &\ (-13/2,-11/2,-9/2,-7/2,-5/2,-3/2,-1/2;-1) \cup (13/2,15/2)\\
= &\ (-15/2,-13/2,-11/2,-9/2,-7/2,-5/2,-3/2,-1/2;-1) \cup {\bf (13/2)}\\
\longrightarrow &\ (-17/2,-15/2,-13/2,-11/2,-9/2,-7/2,-5/2,-3/2,-1/2;-1) \cup (9)
\end{align*}
and we are in Case (c) above.
\end{example}

\subsection{Non-spherical Case} Now we study the case when $\mu_1 > 0$. Then the parameter $(\la_{rel},-s\la_{rel})$ does not have a
$\kappa_0$, or else the regularity condition is violated.
Consider the spherical part of the parameter. It only contains $\kappa_i$ for $i > 0$ and
$\sigma_j$ for $j \geq 0$.  By Lemma \ref{l:bottomb} this spherical parameter yields
indefinite form on $V_{\mathfrak{k}}(1,1,0,\dots,0)$ and $V_{\mathfrak{k}}(0,\dots,0)$,
both are bottom layer in $J(\la_{rel},-s\la_{rel})$. Therefore there cannot be any spherical parameter, and the only unitary case is $(\la_{rel},-s\la_{rel}) = (1/2, -1/2)$.

\subsection{Spin Groups} \label{sec:spinb} In this section, we give a brief idea on how our results can be extended to Spin groups $G=Spin(2n+1,\bC)$. We only
consider genuine representations of $G$, i.e. representations
whose $K-$types have highest weights with
coordinates of the form $\mathbb{N} + \frac12$ only. As $\rho=(m-1/2,\dots
,1/2)$, so $2\la=\{\eta-\rho\}+\rho$ must have
 coordinates of the form $\mathbb{N} + \frac12$ only. The integral system for $\la$ is type A.

We study the case when the lowest
$K-$type of $J(\la_L,\la_R)$ is
$$
Spin = V_{\mathfrak{k}}(\frac12,\dots ,\frac12).
$$
The parameter is
\begin{equation}
  \label{eq:spinparameter}
  \begin{aligned}
&\la_L=&&(1/4,\dots ,1/4)&&+(\nu_1,\dots ,\nu_k,-\nu_k,\dots
,-\nu_1)\\
&\la_R=&&\left(-1/4,\dots ,-1/4\right)&&+(\nu_1,\dots ,\nu_k,-\nu_k,\dots
,-\nu_1)\\
&\text{ or } &&&&\\
&\la_L=&&(1/4,\dots ,1/4)&&+(\nu_1,\dots ,\nu_k,0,-\nu_k,\dots
,-\nu_1)\\
&\la_R=&&(-1/4,\dots ,-1/4)&&+(\nu_1,\dots ,\nu_k,0,-\nu_k,\dots
,-\nu_1)
  \end{aligned}
\end{equation}
The symmetry $\nu_i\longleftrightarrow -\nu_i$ follows from the assumption that
the parameter must be Hermitian.
Since $2\la_L=(\frac12+2\nu_1,\dots ,\frac12-2\nu_1)$ must be regular
integral consisting of half-integers, it follows that
\begin{equation}
  \label{eq:integrality}
    2\nu_i\in \bZ \text{ for all } i,
\end{equation}
satisfying $\nu_i\pm \nu_{j}\ne 0,$ and $\nu_i\ne 0.$

\smallskip
Separate the $\nu_i$ into integers
$\nu_a$  and half-integers $\nu_b.$ The Hermitian property implies
that $\nu_a$ must be conjugate to $-\nu_a$ by the symmetric group, and
similarly for $\nu_b.$

There are two finite
dimensional Hermitian representations $F_a$ and $F_b$ of Type A (with lowest
$K-$types $V_{\mathfrak{u}}(\frac12,\dots,\frac12)$) so that
\begin{equation}
  \label{eq:spinlkt}
J(\la_L,\la_R)=\Ind^G_{GL\times GL}(F_a\otimes F_b).
\end{equation}

The restriction of $V_{\mathfrak{k}}(\frac32,\frac12,\dots ,\frac12)$ to $GL$ contains
$$
V_{\mathfrak{u}}(\frac32,\frac12,\dots ,-\frac12)=V_{\mathfrak{u}} (\frac12,\dots ,\frac12)\otimes V_{\mathfrak{u}}(1,0,\dots,0,-1).
$$
Therefore, as in Proposition \ref{p:lambdar},
the Hermitian form of $J(\la_L,\la_R)$ on the $K-$types $V_{\mathfrak{k}}(\frac32,\frac12,\dots
,\frac12,\frac12)$ and $V_{\mathfrak{k}}(\frac12,\dots ,\frac12)$ is indefinite
 unless $F_a,F_b$ are unitary characters. In the case when there is only $F_a$ or $F_b$ in \eqref{eq:spinlkt},
we obtain the genuine unipotent representation with infinitesimal character given in \eqref{eq:B.2}.


\section{Proof of Theorem \ref{thm:j2} -- Type C}
Let $G = Sp(2m,\mathbb{C})$ and  $K = Sp(2m)$. The $K-$types have highest weights $\eta$
formed of integers only. Since $\rho=(m,\dots ,1),$
$2\la=\{\eta - \rho\} + \rho$ must have positive integer coordinates only.
So $\la$ must have  integers and half integer coordinates only.
Since $\la$ is regular half-integral but not integral, the integral
system determined by $\la$ is type $B \times D.$

\subsection{Spherical Representations}
\begin{proposition} \label{thm:unitc}
Let $\la$ be regular half-integral. The spherical irreducible module $J(\la,\la)$ is unitary if and only if it is unipotent, i.e. the parameter is
$$
  \la=\left(-K_0+\frac12,\dots,-\frac12\right)\quad \text{ or } \quad
  \la=\left(-N_0,\dots,-1\right)
$$
The first representation is the spherical component of the Oscillator
representation attached to the
nilpotent orbit $[2^11^{2N_0-2}]$, and the second case is the trivial
representation attached to $[1^{2N_0}]$.

When not unitary, the form is indefinite on the set of cx-relevant $K-$types with highest weights
$$
CXC:=\{(0,\dots ,0),\quad (1,1,0,\dots ,0),\quad (2,0,\dots ,0)\}.
$$
\end{proposition}
Unlike Types B or D, only (1,1,0,\dots,0), rather than (1,\dots,1,
  0,\dots,0) suffices.
The proof will be given in the next subsection. The unipotent
representation is unitary because, when not the trivial module, it is
the spherical component of the Oscillator representation.

As in the case of Type B, we construct a parabolic subgroup $P(\la)=L(\la)U(\la)$ and an induced module
$I_{P(\la)}$ so that $J(\la,\la)$ is the spherical irreducible
  factor in $I_{P(\la)},$ and the multiplicities of the cx-relevant
  $K-$types coincide in the two modules. Write $\la$ dominant for the standard
  positive system, i.e.
	$$\la = (\dots \la_i\ge \la_{i+1} \ge \dots \geq 0), \quad 2\la_i \in \mathbb{Z}.$$
  Since the parameters we are going to study are obtained by deforming a regular parameter {\it upward},
	we can further assume that all $\la_i$ are positive.
\begin{itemize}
\item[(i)]
If $1/2$ is a coordinate of $\la,$ form the longest
\textit{string}
\begin{equation*}
  \kappa_0=(-K_0 +1/2,\dots ,-1/2)
\end{equation*}
such that all the half-integers starting from $1/2$ to $K_0-1/2$ are
coordinates of  $\la$, but $K_0+1/2$ is not. If the coordinate $1$
occurs, form the longest string
$$\sig_0=(-N_0,\dots ,-1)$$
where $1,\dots ,N_0$ occur as coordinates in $\la,$ but $N_0+1$ does not.
Add a factor of $L(\la)$ of type $
G(K_0+N_0)=Sp(2K_0+2N_0)$ and the spherical irreducible representation with parameter
$$
\begin{pmatrix}
  -K_0+1/2&,\dots ,&-1/2&;&-N_0,&\dots,&-1\\
  -K_0+1/2&,\dots ,&-1/2&;&-N_0,&\dots,&-1
\end{pmatrix}
$$
If $1/2$ is not a coordinate, let $k_1-1/2>0$ be the smallest
half-integer coordinate, and form the longest string
$\kappa_1=(k_1-1/2,\dots ,K_1-1/2)$ increasing by 1, as before. Add a
factor  $GL(K_1-k_1+1),$ and the 1-dimensional representation with parameter
$$
\begin{pmatrix}
  k_1-1/2&,\dots ,&K_1-1/2\\
  k_1-1/2&,\dots ,&K_1-1/2
\end{pmatrix}
$$
to $M(\la)$. Similarly if $1$ does not occur as a coordinate, form $\sig_1=(n_1,\dots ,N_1)$ and add a factor $GL(N_1-n_1+1)$ to the Levi component $M(\la).$
\item[(ii)] Remove the coordinates in Step (i) from $\la$, and repeat
  on the remainder until there
  are no half-integer coordinates left. Since the assumption was that at most one
  coordinate was equal to $1/2,$ only $GL-$factors are created.
\item[(iii)] Repeat Steps (i) and (ii) on the integer coordinates
  until there are none left.
\end{itemize}
The process produces a parabolic subgroup, and an induced module on its Levi component. The Levi
component is
\begin{equation} \label{eq:deformc}
\prod_{i > 0} GL(\sigma_j) \times \prod_{j>0} GL(\kappa_i)\times G(K_0+N_0).
\end{equation}
As in the case of Type B, we are interested in the cases when the \textit{strings} satisfy the properties:
\begin{equation}
  \label{eq:mediumC}
\begin{cases} k_i>2 &\text{if }\ 1/2 \text{ is a coordinate},\\ k_i\ge 2 &\text{otherwise}\end{cases}
\quad \text{and} \quad \begin{cases} n_j> 2 &\text{if }\ 1 \text{ is a coordinate},\\ n_j\ge 2 &\text{otherwise}\end{cases}
\end{equation}
along with the nested condition:
\begin{equation}
  \label{eq:extraC}
\begin{cases} k_{i+1}-K_i\ge 2,\ \text{or}\\
k_i\le k_{i+1}\le K_{i+1}\le K_i,\end{cases} \quad \text{and} \quad
\begin{cases} n_{j+1}-N_j\ge 2,\ \text{or} \\
n_{i}\le n_{i+1}\le N_{i+1}\le N_{i} \end{cases}.
\end{equation}
The main property of the cx-relevant $K-$types is the following Lemma.
\begin{lemma}
  Let $\la$ be such that \eqref{eq:mediumC} and \eqref{eq:extraC} are satisfied. The
  multiplicities of the cx-relevant $K-$types is the same in
  $I_{P(\la)}$ and $J(\la,\la)$.
\end{lemma}
\begin{proof}
 The proof follows the one for the analogous result in Type B. We have
 to show that certain $SL(2)_\al-$operators are isomorphisms. For the
 cx-relevant $K-$types this follows from conditions \eqref{eq:mediumC} and \eqref{eq:extraC}
 and the fact that the coordinates of the highest weights of the $K-$types are $\le 2.$
\end{proof}

\subsection{Proof of Proposition \ref{thm:unitc} -- $\la = \mathbf{\sig_0} \cup \mathbf {\kappa_i}$ or $\mathbf{\kappa_0} \cup \mathbf {\sigma_i}$}  \label{sec:unitc1}
If $\la$ contains only $\sig_0 = (-N_0, \dots, -1)$ or $\kappa_0 = (-K_0 + 1/2, \dots, -1/2)$, the parameter is unitary. So consider $\la = \sig_0 \cup \kappa_i$ or $\kappa_0 \cup \sigma_i$ for $i = 0$ or $1$, and the induced module
$$
\Ind_{GL(K_i)\times G(N_0)}^G\left(\kappa_i \otimes(-N_0,\dots, -1)\right) \quad \text{or}\quad \Ind_{GL(N_i)\times G(K_0)}^G\left(\sigma_i \otimes(-K_0 + 1/2,\dots, -1/2)\right).
$$
If $i = 1$, i.e. $k_1 \geq 3/2$ or $N_1 \geq 2$, then the above induced modules admit deformations
where the multiplicities of all cx-relevant $K-$types coincide with that of $J(\la,\la)$ for $0 \leq t < 1/2$.
If $i = 0$, the deformations still preserve multiplicities of the cx-relevant
$K-$types of the form $V_{\mathfrak{k}}(1,\dots,1,0,\dots,0)$. There are two cases:

\smallskip
\noindent (a) Suppose If $k_i - N_0 > 1$ or $n_i - K_0 \geq 1$ (so that $i = 1$), or equivalently
one has $|n - k| \geq 3/2$ for all $n \in \sigma_i$ and $k \in \kappa_j$,
the deformations on $\kappa_1$ or $\sigma_1$
does not produce new $P(\la)$ for all $t \geq 0$. So by Casimir inequality the form is indefinite on
the trivial and the adjoint $K-$type $V_{\mathfrak{k}}(2,0,\dots, 0)$.

\smallskip
\noindent (b) Otherwise,
At $t=1/2,$ the spherical parameter acquires a new $\sig_1$ or $\kappa_1$. As in Type B, we can apply induction hypothesis
and reduce to the \textit{initial cases} when the
spherical parameter at $t=1/2$ is either the trivial representation, or the spherical
Oscillator representation. These are
$$
(N_0+1/2,\dots , N_1+1/2) \cup (-N_0,\dots ,-1) \quad \text{or} \quad  (K_0,\dots , K_1) \cup (-K_0 + 1/2,\dots , -1/2)
$$
The argument for type B applies. At $t=1/2$ there is another factor
\begin{equation} \label{eq:nonunitc}
\begin{aligned}
&\pmat
&K_0+1/2&-K_0+1/2& K_1+1/2&\dots &K_0+3/2&K_0-3/2&\dots &1/2\\
&K_0-1/2&-K_0-1/2&K_1+1/2&\dots &K_0+3/2&K_0-3/2&\dots &1/2
\epmat\\
&\text{ respectively }\\
&\pmat
&N_0+1&-N_0& N_1+1&\dots &N_0+2&N_0-1&\dots &1\\
&N_0-1&-N_0-1&N_1+1&\dots &N_0+2&N_0-1&\dots &1
\epmat
\end{aligned}
\end{equation}
The $K-$types $V_{\mathfrak{k}}(2,0,\dots ,0)$ and $V_{\mathfrak{k}}(1,1,0,\dots ,0)$  are
bottom layer for the parameter in \eqref{eq:nonunitc}, and the form is indefinite. In this case one
can in fact show that at $t=0$ the form is indefinite on
$V_{\mathfrak{k}}(1,1,0,\dots ,0)$ and $V_{\mathfrak{k}}(0,\dots ,0).$ The reason is that one
can deform the string $\kappa_1$ or $\sig_1$ all the way to a place
where the module is unitarily induced irreducible, and $V_{\mathfrak{k}}(2,0,\dots
,0)$ occurs with full multiplicity in the spherical irreducible
module. So its sign must be the same as that of $V_{\mathfrak{k}}(0,\dots ,0).$ Therefore,
$J(\la,\la)$ has
indefinite forms on $V_{\mathfrak{k}}(1,1,0,\dots ,0)$ and $V_{\mathfrak{k}}(0,\dots ,0)$.

\begin{remark} \label{rmk:generalc}
More generally, if $\la = \sigma_i \cup \kappa_j$ satisfies
$k_j \leq N_i+1 \leq K_j$ or $n_i \leq K_j \leq N_i$, i.e. there
are $n \in \sigma_i$ and $k \in \kappa_j$ such that $|n - k| = 1/2$,
then one can deform both strings $\sigma_i$, $\kappa_j$ downwards simultaneously
$$\sigma_i \cup \kappa_j \mapsto \sigma_i \cup \kappa_j - (t, \dots, t),$$
until it reaches Case (b) above. Then one can conclude that $J(\la,\la)$ has
indefinite forms on $V_{\mathfrak{k}}(1,1,0,\dots ,0)$ and $V_{\mathfrak{k}}(0,\dots ,0)$.
\end{remark}

\subsection{Proof of Proposition \ref{thm:unitc} -- Other Strings} We do an induction, downward on the number of strings, upward on the length of the parameter, as in type B. The claim is that if there is a string $\kappa_1$ or $\sigma_1,$ the spherical module cannot be unitary.

For $i > 0$, let $\xi=(k_i-1/2,\dots ,K_i-1/2)$ or $(n_i,\dots ,N_i)$ be a string. Deform upward $\xi_t = (k_i-1/2+t,\dots , K_i-1/2+t)$ or $(n_i+t,\dots ,N_i+t).$ The signatures and multiplicities of the all cx-relevant $K-$types do not change for $0\le t<1/2.$ At $t=1/2,$ one of several cases may occur:
\begin{enumerate}
\item[(a)] There is no $\xi$, that is, $\la = \kappa_0 \cup \sig_0$.
	We have dealt with this in the previous section.
\item[(b)] {$P(\la_{1/2}) = P(\la_{0})$}.
  Continue deforming upwards. If no change occurs as $t\to\infty$ (this includes Case (a) in Section \ref{sec:unitc1}), the form is indefinite on $V_{\mathfrak{k}}(0,\dots ,0)$
  and the adjoint $K-$type $V_{\mathfrak{k}}(2,0,\dots, 0)$.
\item[(c)] {$P(\la_{1/2}) \ne P(\la_{0})$}. Then we are in the setting of Remark \ref{rmk:generalc},
and the form is indefinite on $V_{\mathfrak{k}}(0,\dots ,0)$
  and $V_{\mathfrak{k}}(1,1,\dots, 0)$.
\end{enumerate}
The cases when indefiniteness is first detected on the
$K-$type $V_{\mathfrak{k}}(2,0,\dots ,0)$ rather than $V_{\mathfrak{k}}(1,1,0,\dots,0)$
is when the entries of two different strings in $\la$ differ by at least $1$.
For example, this holds for the strings $\la = (21/2, 23/2) \cup (8,9) \cup (7/2,9/2,11/2)$.

\subsection{Non-spherical Case}\label{sec:nonshc} Consider the case
$\mu_1=1$ and the parameter contains
$
\begin{pmatrix}
  1/2\\-1/2
\end{pmatrix}.
$
As before, there cannot be a $\kappa_0$ present.
The fundamental cx-relevant $K-$types for the spherical parameter
produce bottom layer $K-$types.  We are reduced to the cases when these bottom layer $K-$types
do not detect non-unitarity. By the last paragraph in the previous section,
this is the case when there is a $\kappa_i, \sig_j$ with $i,j > 0$
in the spherical parameter
deforming to $\infty$. The case when there is only $\kappa_1=(3/2,\dots , K_1-1/2)$ in the spherical parameter
gives a unitary representation.
We are reduced to the case when there is another string
$\kappa_i\ge 5/2$ and/or $n_j\ge 2$ deforming to $\infty$. The $K-$types
$$
V_{\mathfrak{k}}(1,0,\dots ,0)\quad V_{\mathfrak{k}}(2,1,0,\dots ,0)
$$
occur with the same multiplicities in the unitarily induced module from
$GL(1)\times G(\mu_0)$ with $J(\la^0,\la^0)$ on the $G(\mu_0)-$factor,
and in $J(\la,-s\la).$  The form is indefinite
on these $K-$types, since they restrict to
$K\cap M-$types for which the form on $J(\la^0,\la^0)$ is indefinite.

\section{Proof of Theorem \ref{thm:j2} -- Type D}
Let $G = SO(2m,\mathbb{C})$ and $K = SO(2m)$. The $K-$types have highest weight with
integer coordinates only.  Since $\rho = (m-1, \dots, 1,0)$, it
follows that $2\la=\{\eta - \rho\} + \rho$ has integer coordinates
only. So $2\la$ is  regular integral it has integer coordinates only. Since $\la$ is not assumed integral, its coordinates are integers and half integers, and the integral system is
of type $D\times D.$

\subsection{Spherical Representations}
\begin{proposition} \label{thm:unitd}
Let $\la$ be regular half-integral. The spherical irreducible module $J(\la,\la)$ is unitary if and only if it is unipotent, i.e.
$$
\la=\left(-K_0+\frac12,\dots,-\frac12; -N_0+1,\dots, -1,0 \right)\quad\text{ satisfying } \quad N_0 \ge K_0.
$$
When $K_0>0,$ the representation is attached to the
nilpotent orbit $[3^12^{2K_0}1^{2N_0-2K_0-1}]$.

When $K_0=0,$ the nilpotent orbit is the  trivial one.

When not unitary, the form is indefinite on the set of cx-relevant $K-$types with highest weights
$$
CXD:=\{ (0,\dots ,0),\quad (1,\dots ,1,0,\dots ,0),\quad
(2,0,\dots ,0)\}.
$$
\end{proposition}
The proof will take up most of the next few subsections.
The unipotent representations are unitary because they
can be realized via  the dual pair correspondence in the stable range, as
$\Theta(\mathrm{triv}_{Sp})$, with the pair $Sp(2K_0,\bC) \times
SO(2K_0+2N_0,\bC)$ and one of the components of the the Oscillator representation on the
$Sp-$factor.

\medskip
As in Type B and C, we construct a parabolic subgroup $P(\la) = L(\la)U(\la)$ and an induced module
$I_{P(\la)}$ for each $\la$ dominant for the standard
  positive system, i.e.
  $$
  \la = (\dots \la_i\ge \la_{i+1},\dots \ge \la_{m-1}\ge |\la_m|\ge 0), \quad 2\la_i\in\mathbb{Z}.
  $$
	
\begin{itemize}
\item[(i)]
If $1/2$ is a coordinate of $\la,$ form the longest
\textit{string}
\begin{equation*}
  \kappa_0=(-K_0 +1/2,\dots ,-1/2)
\end{equation*}
such that all the half-integers staring from $1/2$ to $K_0-1/2$ are
coordinates of  $\la$, but $K_0+1/2$ is not. If the coordinate $0$
occurs, form the longest string
$$
\sig_0=(-N_0+1,\dots ,-1,0)
$$
where
$N_0-1$ is the largest integer coordinate that occurs in $\la,$ but
$N_0$ does not.
Add a factor  of type $G(K_0+N_0)=SO(2K_0+2N_0)$ to $L(\la)$, and the spherical irreducible representation with parameter
$$
\begin{pmatrix}
  -K_0+1/2&,\dots ,&-1/2&;&-N_0+1,&\dots,&-1&0\\
  -K_0+1/2&,\dots ,&-1/2&;&-N_0+1,&\dots,&-1&0
\end{pmatrix}.
$$
If $1/2$ is not a coordinate, let $k_1-1/2>0$ be the smallest
half-integer coordinate, and form the longest string
$\kappa_1=(k_1-1/2,\dots ,K_1-1/2)$ going up by one as before. Add a
factor  $GL(K_1-k_1+1),$ to $L(\la),$ and the 1-dimensional representation with parameter
$$
\begin{pmatrix}
  k_1-1/2&,\dots ,&K_1-1/2\\
  k_1-1/2&,\dots ,&K_1-1/2
\end{pmatrix}.
$$
 Similarly if $0$ does not occur as a coordinate, form
 $\sig_1=(n_1,\dots ,N_1)$ and add a factor $GL(N_1-n_1+1)$ to the
 Levi component $L(\la).$
\item[(ii)] Remove the coordinates in Step (i) from $\la$, and repeat
  on the remainder of half integer coordinates until there
  are no half-integer coordinates left. Similarly for the integer
  coordinates. Since the regularity assumption implies that at most one
  coordinate can be  equal to $1/2,$ and at most one coordinate equal
  to $0,$  only $GL-$factors are created.
\end{itemize}
The process produces a parabolic subgroup, and an irreducible module on its Levi component. The Levi
component is
\begin{equation} \label{eq:deformb1}
\prod_{i > 0} GL(\sigma_j) \times \prod_{j>0} GL(\kappa_i)\times G(K_0+N_0).
\end{equation}
The parameters $\la$ we are going to study satisfy:
\begin{equation}
  \label{eq:mediumD}
\begin{cases} k_i>2 &\text{if }\ 1/2 \text{ is a coordinate},\\ k_i\ge 2 &\text{otherwise}\end{cases}
\quad \text{and} \quad \begin{cases} n_j> 1 &\text{if }\ 0 \text{ is a coordinate},\\ n_j\ge 1 &\text{otherwise}\end{cases}
\end{equation}
and the nested condition
\begin{equation}
  \label{eq:extraD}
\begin{cases} k_{i+1}-K_i\ge 2,\ \text{or}\\
k_i\le k_{i+1}\le K_{i+1}\le K_i,\end{cases} \quad \text{and} \quad
\begin{cases} n_{j+1}-N_j\ge 2,\ \text{or} \\
n_{i}\le n_{i+1}\le N_{i+1}\le N_{i} \end{cases}.
\end{equation}
The main property of the cx-relevant $K-$types is the following Lemma.

\begin{lemma} \label{prop:bij} Assume that the strings of $\la$ satisfy
\eqref{eq:mediumD} and \eqref{eq:extraD}.
The multiplicities of the  cx-relevant $K-$types in $I_{P(\la)}$
coincide with those in $J(\la,\la).$
\end{lemma}
\begin{proof} The proof follows the analogous result for Type B. In
    this case all cx-relevant $K-$types are petite/single petaled.  This is because
   $(\check\al,\la)\le 3$ for all roots.
\end{proof}

As in Type B, we have a necessary condition on the spherical parameter:
\begin{lemma}\label{l:bottomd}
If $J(\lambda,\lambda)$ is unitary, then the string $\sigma_0 = (-N_0+1, \dots, 1,0)$ must appear in $(\lambda,\lambda)$.
\end{lemma}
\begin{proof}
  The coordinates on the spherical part of $I_1$ in Equation \eqref{eq:inclusion}
	are all $\ge 1/2.$ As in Lemma \ref{l:bottomb}, the irreducible representation $J(\lambda,\lambda)$ has indefinite form on the adjoint $K-$type $V_{\mathfrak{k}}(1,1,0,\dots ,0)$ and the trivial $K-$type.
\end{proof}

\subsection{Proof of Proposition \ref{thm:unitd} -- $\mathbf{\la=\kappa_0\cup\sig_0}$} \label{sec-bn>k1}
The case when $N_0\ge K_0$ is unitary. So assume
\begin{equation}
  \label{zerostrings1}
 \la = \kappa_0\cup \sigma_0\quad \text{satisfying }\quad  K_0>N_0.
\end{equation}

By Lemma \ref{l:bottomd}, we assume $N_0>0$, and let
$$
Ind(\la_t):=\Ind^G_{GL(\kappa_0)\times G(N_0)}\left((1/2+t,\dots, K_0-1/2+t) \otimes (-N_0+1, \dots,-1,0) \right).
$$
The multiplicities of all cx-relevant $K-$types in $Ind(\la_t)$ and
$J(\la,\la)$ still coincide for small $t$. This is as before: $Ind(\la)$ is
a homomorphic image of $\Ind_B^G(\la,\la)$, and
the intertwining operator changing $(1/2,\dots ,K_0-1/2)$ to
  $(-K_0+1/2,\dots ,-1/2)$ involves only $(\check\al,w\la)$ which are half-integers or  $\ge 2$:
  $$
    \mat{1/2,0}{1/2,0}\mapsto \mat{0,-1/2}{1/2,0}\quad \text{ or } \quad
    \mat{1/2,3/2}{1/2,3/2}\mapsto \mat{-3/2,-1/2}{1/2,3/2}
  $$
  depending whether $K_0$ is even or odd. In the first case, the
  $SL(2)-$intertwining operator is an isomorphism, in the other case
  the kernel of the intertwining operator has lowest $K-$type $(2,2).$
  So the intertwining operator is an isomorphism on the cx-relevant $K-$types.

\smallskip	
The signatures (and multiplicities) of the fundamental cx-relevant $K-$types of $Ind(\la_t)$ do not change for
$0\le t< 1/2.$ At $t=1/2,$ the parameter is
\begin{equation*}
\la_{1/2}= (1, \dots, K_0 ; -N_0+1, \dots, -1,0) = (-K_0,\dots, -1,0)\cup (1,\dots ,N_0-1).
\end{equation*}
As in the case in Type B, $J(\la_{1/2},\la_{1/2})$ and $Ind(\la_{1/2})$ are different on the level of
fundamental $K-$types, and $Ind(\la_{1/2})$ has another factor with parameter
\begin{equation}
  \label{eq:factor}
  \begin{pmatrix}
  -N_0+1&\dots &N_0&;&-K_0&\dots &-N_0-1\\
  -N_0&\dots &N_0+1&;&-K_0&\dots &-N_0-1
\end{pmatrix}
\end{equation}
and lowest $K-$type $\mu_0=(\underset{2N_0}{\underbrace{1,\dots ,1}},0,\dots 0)$.

If $K_{0}-N_{0}$ is odd, the factor is not Hermitian, and there is another
factor which is Hermitian dual to it, whose parameter
$-K_{0},\dots,-N_{0}-1$ is changed to its negative in both
$\lambda_{L}$ and $\lambda_{R}$. In this case, the signature is indefinite on a single $K$-type $\mu_0$. When $K_{0}-N_{0}>0$ is even, the signature
is indefinite on $\mu_{0}$ and
$\mu_{1}=(\underset{2N_{0}+2}{\underbrace{1,\dots,1}},0,\dots0)$.

In both cases, $Ind(\la_{1/2})$, and hence $Ind(\la)$ and $J(\la,\la)$,
has indefinite signature on the fundamental cx-relevant $K-$types.

\subsection{Proof of Proposition \ref{thm:unitd} -- Other Strings}  \label{sec:dother}
We follow the reasoning for type B. We do a
downward induction on the length of $\la$, and the number of
strings. The case when there are no strings other than
$\kappa_0,\sig_0,$ was dealt with in the previous section. As in Type B, there are three {\it initial cases}:
\begin{itemize}
\item[(a)] There is a string $\kappa_i$ or $\sig_j$ with $i,j>0$ such that
  $P(\la_t)$ does not change as $t\to\infty$,
\item[(b)] The strings are
  $$
  (-K_0 +1/2, \dots, -1/2; -N_0+1, \dots, 1,0),\quad \text{with}\ K_0 > N_0$$
  as in the previous section.
\item[(c)] The strings are $\la = \kappa_0 \cup \sig_0 \cup \xi$, where
$$
\xi = (K_0,\dots , K_1-1)\quad \text{ or }\quad (N_0-1/2,\dots , N_1-3/2),
$$
so that the deformation of $\xi$ to $t=1/2$ yields a unitary spherical
module. This means that $N_0 \geq K_1$ in one case, $N_1 \geq K_0$ in  the
other case.
\end{itemize}
As in Type B, Case (a) and (b) yield indefinite signatures on the trivial and adjoint $K-$type $V_{\mathfrak{k}}(1,1,0,\dots ,0)$.
And Case (c) yields indefinite form on $V_{\mathfrak{k}}(1,1,0,\dots ,0)$ and  $V_{\mathfrak{k}}(2,0,\dots,0)$.

\subsection{Non-spherical case}\label{sec:nsphd} If $\mu_1 > 0$, the parameter contains
$$
\begin{pmatrix}
  1&0\\0&-1
\end{pmatrix}
\quad \text{ or }\quad
\begin{pmatrix}
  1/2\\-1/2
\end{pmatrix}.
$$
Suppose $\disp{\begin{pmatrix}
    1&0\\0&-1\end{pmatrix}}$ occurs. If the parameter has no spherical
part, there is nothing to be done; the parameter is unitary. If the
parameter has a spherical part, there cannot be a $\sig_0$ or else the regularity of the parameter is violated.  Lemma \ref{l:bottomd} implies that the Hermitian form is indefinite on
the trivial and adjoint $K-$types. Both are bottom layer if the lowest $K-$type has coordinates greater than one.

The proof of the claim is reduced to the case when
the non-spherical parameter is exactly
$
\begin{pmatrix}
  1/2\\-1/2
\end{pmatrix}
$, and the spherical parameter contains a $\sigma_0$.

The only case when the bottom layer $K-$type does
not detect non-unitarity is in Case (c) in Section \ref{sec:dother}, which occurs when
there is no $\kappa_0$ (due to regularity of $\la$), and a string $\kappa_i$ ($i > 0$) in the spherical parameter such that it is deformed to $\xi = (N_0 - 1/2, \dots, N_1 - 3/2)$.
The case when the spherical part is exactly $\sigma_0 \cup \kappa_1$ with $\kappa_1 = (3/2,\dots ,K_1-1/2)$  and $N_ 0 \geq K_1$ is unitary. Otherwise,
we have $\kappa_1 = (3/2,\dots ,K_1-1/2)$  and $N_ 0 < K_1$ which is not unitary on the level of bottom layer $K-$types by Case (b) above,
or there is a string $\kappa_i$ in the spherical parameter satisfying $k_i-1/2\ge 5/2$.
The fact that $k_i-1/2\ge 5/2$ implies that the $K-$types
$$V_{\mathfrak{k}}(2,1,0,\dots ,0)\quad \text{and}\quad V_{\mathfrak{k}}(1,1,1,0,\dots ,0)$$
occur with the same multiplicity in $J(\la,-s\la)$ and in the unitarily
induced module from the spherical part. Since their restrictions to the Levi component contain
$K-$types with indefinite form, the conclusion follows.

\subsection{Spin Groups} \label{sec:spind}  As in Section \ref{sec:spinb}, we study genuine representations
of $G=Spin(2n,\bC)$ in this section. The $K-$types have highest weights with
coordinates in $\mathbb{N} \cup + \frac12$ only, except the last coordinate can be $-\frac12$. As already mentioned, $\rho=(m-1,\dots
  ,1,0)$, so $2\la=\{\eta-\rho\}+\rho$ must have
  coordinates of the form $\mathbb{N} + \frac12$ only (the last coordinate can be $-\frac12$). The integral system for $\la$ is type A.

We consider the case wen the lowest $K-$type of $J(\la_L,\la_R)$ is
$Spin^{\pm} = V_{\mathfrak{k}}(\frac12,\dots \frac12,\pm\frac12)$. Using the same arguments as
in Section \ref{sec:spinb}, all such irreducible modules must be of the form
\begin{equation}
  \label{eq:spinlkt1}
J(\la_L,\la_R)=Ind_{GL\times GL}^G(F_a\otimes F_b).
\end{equation}
Unless $F_a,F_b$ are one dimensional, the form is
indefinite on the lowest $K-$type $V_{\mathfrak{k}}(\frac12,\dots,\frac12)$ and `adjoint' $K-$type $V_{\mathfrak{k}}(\frac32,\frac12,\dots
,\frac12,\mp \frac12).$ In the case when there is only $F_a$ in \eqref{eq:spinlkt1}, and the $GL$ corresponds to either one of the two subroot system of $D_n$,
one obtains the genuine unipotent representations with infinitesimal character given in \eqref{eq:D.2}.

 \section{A Positivity Result}
In this section, we sharpen the results in Section 5.4-5.6 in
\cite{BP}. We investigate the PRV-components of $\pi_u \otimes V_{\mathfrak{k}}(\rho)$
when $\pi_u$ is a unipotent representation with half-integral
regular infinitesimal character for a classical group.

By
\cite[Section 5.4-5.6]{BP}, all $\pi_u \in \widehat{G}^d$ for Type $B_n$, while
for Type $C_n$ and $D_n$
$\pi_u^{even/odd} \in \widehat{G}^d$ if and only if $n$ is
even/odd. Moreover, the spin-lowest $K-$type is
unique for all such $\pi_u$'s
(this will be verified in Proposition \ref{prop-positive-result} below).

\medskip
Since the $K-$types of $\pi_u$ are multiplicity free, Theorem \ref{thm:main}
holds for all $\pi_u \in \widehat{G}^d$. In order to prove
Theorems \ref{thm:main1} and \ref{thm:main} for general $\pi \in \widehat{G}^d$,
we need the following refinement of the results in \cite{BP}:
\begin{proposition}\label{prop-positive-result}
Let $G$ be a connected complex classical simple Lie group and $\pi_u = J(\la,-s\la)$ be a 
unipotent representation given in Theorem \ref{thm:3}. If $\pi_u \in \widehat{G}^d$,
then there is a unique $K-$type $V_{\mathfrak{k}}(\eta)$ in $\pi_u$
such that $\delta := \{\eta - \rho\} = 2\la - \rho$ realizes the
minimum of $\{\eta'-\rho\}$ over the $K-$spectrum of $\pi_u$.
Furthermore,
\begin{equation}
\pi_u \otimes V_{\mathfrak{k}}(\rho) = V_{\mathfrak{k}}(\delta) \oplus \bigoplus_{\delta'\neq \delta} m_{\delta'}V_{\mathfrak{k}}(\delta'),
\end{equation}
where $m_{\delta'}$ are positive integers  and
\begin{equation}\label{delta-deltaprime}
\delta' = \delta + \sum_{i=1}^{l} m_i \alpha_i, \mbox{ satisfying } m_i\in\bZ_{\geq 0}.
\end{equation}
If $\pi_u \notin \widehat{G}^d$, then all $\widetilde{K}-$types of $\pi_u \otimes V_{\mathfrak{k}}(\rho)$ have extremal weights of the form \eqref{delta-deltaprime}
for $\delta$ with norm strictly greater than $\|2\la - \rho\|$.
\end{proposition}

\begin{proof}
The statement is obvious when $\pi_u = triv$ is the trivial representation. So we assume $\pi_u$ is not trivial from now on.
Let $\eta'$ be any $K-$type of $\pi_u$ other than a spin-lowest $K-$type $\eta$. Put $\delta':=\{\eta'-\rho\}$. In view of Theorem \ref{thm-PRV}, it suffices to prove that
\eqref{delta-deltaprime} holds for $\delta$ and $\delta'$.

\textbf{Type $\mathbf{B_n}$:} Let
$V_{\mathfrak{k}}(\eta') = V_{\mathfrak{k}}(\alpha_1, \alpha_1, \dots, \alpha_a, \alpha_a, \underbrace{0, \dots, 0}_{b-a})$
be a $K-$type in $\pi_u$. Since
$\rho=(n-1/2,n-3/2,\dots ,1/2),$ the PRV-component  $\delta'$ is, up  to the action of $W(B_n)$,
\begin{equation}\label{deltaprime-Bn}
x\delta'=(n-2a-1/2, n-2a-3/2, \dots, 1/2, B_1,\dots, B_{2a})
\end{equation}
The minimum is attained when all $B_i=1/2,$ and this can only be
achieved from
$$
\eta=(n-1,n-1,n-3,\dots , n-2a-1,n-2a-1,0,\dots ,0).
$$
It follows that
\begin{equation}
  \label{delta-Bn}
 \delta=(n-2a-1,\dots, 1/2,1/2,\dots ,1/2).
\end{equation}
Any other $K-$type must give rise to a $\delta'$ with at least $B_1\ge 3/2,$ and $B_i\ge 1/2.$ The difference
$x\delta'-\delta$, from \eqref{deltaprime-Bn} and  \eqref{delta-Bn}, is a sum
of short positive roots; on each nonzero coordinate it is $B_i-1/2$
times the corresponding short root. The difference $x\delta'-\delta'$,
as in \eqref{deltaprime-Bn}, is clearly a sum of positive roots since
the two are conjugate, and $\delta'$ is dominant.

\textbf{Type $\mathbf{C_n}$:} Here
$V_{\mathfrak{k}}(\eta') = V_{\mathfrak{k}}(2k,0,\dots,0)$ or $V_{\mathfrak{k}}(2k+1,0,\dots,0)$ and
$\rho = (n,n-1, \dots, 1)$.
The PRV-component is, up to $W(C_n)$,
\begin{equation}\label{deltaprime-Cn}
\delta' = (n-1, n-2, \dots,1, |n-2k|)\quad \text{ or } \quad (n-1,\dots , |n-2k-1|).
\end{equation}
The minimum is attained at $k=\frac{n}{2}$ if $n$ is even,
$k=\frac{n\pm 1}{2}$ if $n$ is odd. Thus
\begin{equation}\label{delta-Cn}
\delta=(n-1, n-2, \dots, 1, 0)\quad \text{ or }\quad (n-1, n-2, \dots, 1, {\bf 1}).
\end{equation}
The argument for Type B applies to derive the conclusion in the
statement of the Proposition.

Also, since $\delta + \rho$ is equal to $2\la = (2n-1,\dots,3,1)$ if and only if
$\delta=(n-1, n-2, \dots, 1, 0)$, it also follows that
$H_D(\pi^{even})\ne 0$ and $H_D(\pi^{odd})=0$ if $n$ is
even, and the reverse is true if $n$ is odd.

\textbf{Type $\mathbf{D_n}$:} We only consider $b > a > 0$ and omit the easier case when $b=a$.
Here
$$
V_{\mathfrak{k}}(\eta')=V_{\mathfrak{k}}(\alpha_1,  \dots,  \alpha_{2a}, \underbrace{0, \dots, 0}_{b-a}),
$$
where $\sum_i \al_i$ is even/odd if $\pi_u^{even/odd}$ is being considered,
and $\rho = (n-1,\dots,1,0)$.
Then the PRV-component, up to the action of
$W(D_n)$, is
\begin{equation}\label{deltaprime-Dn}
\delta' = (n-2a-1,  \dots, 1, 0, |n-1-\al_1|,\dots ,|n-2a-\al_{2a}|)
\end{equation}
Even though  $W(D_n)$ only
allows an even number of sign changes, in the case $b>a$ there is a
coordinate equal to $0,$ so we can change all coordinates to $\ge 0.$ As in type C,
$$
\delta = (n-2a-1,\dots, 1,0,\dots ,0)\quad \text{ or } (n-2a-1,\dots
,1,{\bf 1},0,\dots ,0),
$$
and $H_D(\pi_u^{even}) \neq 0$ if and only if $\delta$ take the first value.
We omit further details which are as in Types B and C.
\end{proof}


The above proposition demonstrates a strong positivity result on the $\widetilde{K}$-types
appearing in the tensor product decomposition of $\pi_u \otimes V_{\mathfrak{k}}(\rho)$ for
unipotent representations $\pi_u$.
In fact, similar calculations have been carried out for other irreducible
unitary representations, and so far there are no counter-examples to the following conjecture,
which sharpens Conjecture \ref{conj:multone} in view of Proposition \ref{prop-D-spin-lowest}:
\begin{conjecture}\label{conj-positive-result}
Proposition \ref{prop-positive-result} holds for any $\pi\in \widehat{G}^d$.
\end{conjecture}

\section{Proof of Theorems \ref{thm:main1} and \ref{thm:main}}
We prove Theorems \ref{thm:main1} and \ref{thm:main} by sharpening the results in Section 2.2 of \cite{BP}. To conform to the
notation in that section, write $\pi_\fk m=J(\la_\fk m,-s\la_\fk m)$
for a unitary representation such that the center of $M$
acts trivially.  In particular, when $\pi_\fk m$ is
  1-dimensional, it is the trivial representation. This case occurs in
  all classical types, and is the only case for type A and Spin
  groups. We assume that $\la_\fk m$ is regular integral dominant
for a positive system $\Delta_M,$ and $\la$ is regular
half-integral. The relations
\begin{equation}
  \label{eq:uparameters}
\begin{aligned}
&\la_\fk m +s\la_\fk m=\mu_\fk m,\quad &&2\la_\fk m=\mu_\fk m+\nu_\fk m,
\\
&\la_\fk m-s\la_\fk m=\nu_\fk m,\quad &&2s\la_\fk m=\mu_\fk m-\nu_\fk m,\\
  &\la=\xi/2+\la_\fk m,\quad &&\mu=\xi +\mu_\fk m,\\
  &s\la=\xi/2+s\la_\fk m,\quad &&\nu=\nu_\fk m.
\end{aligned}
\end{equation}
hold, with $s\in W_M\subset W$. The unitary character $\xi$ can be assumed
dominant for a choice of $\Delta(\fk n)$. We denote $\Delta=\Delta_M\cup\Delta(\fk n).$ However $\lambda$ may
not be dominant for $\Delta$, so let $\Delta'$ be the positive
system for which $\lambda$ is dominant. Since $\lambda$ is dominant for $\Delta_M$,
$$
\Delta_M \subset \Delta'\cap \Delta.
$$
For $\pi_\fk m$, we assume in addition that
\begin{enumerate}
\item[(i)] $\pi_\fk m$ is unitary,
\item[(ii)] $\la$ is regular half-integral,
\item[(iii)]  $\pi_\fk m\otimes V_{\mathfrak{k} \cap \mathfrak{m}}(\rho_\fk m)$ contains only  $\widetilde{K\cap M}-$types of the form
  $$
  \delta_M'=\delta_M+\sum_{\gamma\in \Delta_M}
  m_\gamma\gamma,\quad m_\gamma\in\bN,\quad \text{ with }\quad \delta_M=2\la_\fk m-\rho_\fk m
  $$
\end{enumerate}
By Proposition \ref{prop-positive-result}, this covers all $\pi_u$ in Theorem \ref{thm:unitarydual}
with $H_D(\pi_u) \neq 0$ for classical types, and the case of $\pi_{u} = triv$ for Spin groups. 

By Proposition \ref{prop-D-spin-lowest}, the only
$\widetilde{K}$-type that can appear in the Dirac cohomology of $\pi$ must have extremal weight
$\tau' := 2\lambda - \rho'$, where $2\rho'$ is the sum of all positive roots in $\Delta'$.
By abuse of notations, we write $V_{\mathfrak{k}}(\tau')$ as the $\widetilde{K}$-type with extremal weight
$\tau'$. The relation
\begin{equation}
  \label{eq:tautotauM}
  \begin{aligned}
\tau'=&2\la-\rho'=\xi+\mu_\fk m+\nu_\fk m -\rho'=\xi +2\la_\fk m
-\rho'= \xi+ \delta_M +\rho_\fk m-\rho'=\\
&\xi +\delta_M -w_\fk m\rho+\rho_\fk n-\rho'=\delta_M+(\xi+\rho_\fk n) -(w_\fk m \rho+\rho'),
  \end{aligned}
\end{equation}
because $w_\fk m\rho=-\rho_\fk m +\rho_\fk n.$ Furthermore,
\begin{equation}
  \label{eq:rholessrhoprime}
w_\fk m\rho+\rho'=\sum_{\beta\in\Delta'\cap\Delta(\fk n)} \beta
\end{equation}
Continuing with the proof of \cite[Theorem 2.4]{BP} in Section 2.2,
\begin{equation} \label{eq-mult}
\begin{aligned}
\left[\pi \otimes V_{\mathfrak{k}}(\rho) : V_{\mathfrak{k}}(\tau')\right] &=  [\pi : V_{\mathfrak{k}}(\tau') \otimes V_{\mathfrak{k}}(\rho)] \\
&= [\pi_{\fk m} \otimes \mathbb{C}_{\xi} : V_{\mathfrak{k}}(\tau')|_M \otimes V_{\mathfrak{k}}(\rho)|_M] \\
&= [\pi_{\fk m} \otimes \mathbb{C}_{\xi} \otimes V_{\mathfrak{k}}(\rho)|_M : V_{\mathfrak{k}}(\tau')|_M] \\
&= [\pi_{\fk m} \otimes \mathbb{C}_{\xi} \otimes (V_{\mathfrak{k} \cap \mathfrak{m}}(\rho_{\mathfrak{m}}) \otimes \mathbb{C}_{\rho_{\mathfrak{n}}} \otimes {\bigwedge}^{\bullet}\mathfrak{n}^*):
V_{\mathfrak{k}}(\tau')|_M] \\
&= [\pi_{\fk m} \otimes V_{\mathfrak{k} \cap \mathfrak{m}}(\rho_{\mathfrak{m}}) \otimes \mathbb{C}_{\xi + \rho_{\mathfrak{n}}} \otimes {\bigwedge}^{\bullet}\mathfrak{n}^*:
V_{\mathfrak{k}}(\tau')|_M].
\end{aligned}
\end{equation}
The penultimate step above uses \cite[Lemma 2.3]{BP}, and that
${\bigwedge}^{\bullet}\mathfrak{n}^*$ consists of weights of the form
$\displaystyle{-\sum_{\alpha \in S}} \alpha$, where $S$ is a subset of the roots in $\Delta(\fn)$.

\begin{proposition}\label{prop-branching-Dirac}
Let $\pi = {\rm Ind}_{M}^G(\mathbb{C}_{\xi} \otimes \pi_{\fk m})$ be an irreducible, unitary
representation with $\pi_\fk m$ satisfying (i)-(iii).  Then
\begin{equation}\label{eq:mult}
[\pi_{\fk m} \otimes V_{\mathfrak{k} \cap \mathfrak{m}}(\rho_{\mathfrak{m}}) \otimes \mathbb{C}_{\xi + \rho_{\mathfrak{n}}} \otimes {\bigwedge}^{\bullet}\mathfrak{n}^*:
V_{\mathfrak{k}}(\tau')|_M] = [\pi_\fk m \otimes V_{\mathfrak{k} \cap \mathfrak{m}}(\rho_\fk m): V_{\mathfrak{k} \cap \mathfrak{m}}(\delta_M)].
\end{equation}
(Recall that $H_D(\pi_\fk m)$ is either zero or a multiple of
$V_{\mathfrak{k} \cap \mathfrak{m}}(\delta_M)$).
\end{proposition}
\begin{proof}
We use (iii); the fact that $\pi_\fk m\otimes V_{\mathfrak{k} \cap \mathfrak{m}}(\rho_\fk m)$ is a sum
of $\widetilde{K\cap M}-$types of the form
$$
\delta_M'=\delta_M+\underset{\gamma\in \Delta_M}{\sum}
m_\gamma\gamma.$$

Tensoring with
$\mathbb{C}_{\xi + \rho_{\mathfrak{n}}} \otimes {\bigwedge}^{\bullet}\mathfrak{n}^*$,
the $\widetilde{K\cap M}$-types that appear  must have highest weights of the form
$$
\delta_M' + \xi + \rho_{\mathfrak{n}} - \sum_{\alpha \in S} \alpha$$
for some $S \subseteq \Delta(\fn)$.

Combining
the arguments above, any $\widetilde{K\cap M}$-type appearing on the left module in \eqref{eq:mult}
must have highest weights of the form
\begin{equation} \label{otherktypes}
\begin{aligned}
&\delta_M' + \xi + \rho_{\mathfrak{n}} - \sum_{\alpha \in S} \alpha \\
=&\left(\delta_M + \sum_{\gamma \in \Delta_M,\ m_{\gamma} \geq 0} m_{\gamma}\gamma\right) + \xi + \rho_{\mathfrak{n}} -
\left(\sum_{\alpha \in S \cap \Delta'}
															\alpha + \sum_{\beta \in S \cap (-\Delta')}
															\beta\right) \\
=&\left(\delta_M + \sum_{\gamma \in \Delta_M,\ m_{\gamma} \geq 0} m_{\gamma}\gamma\right) + \xi + \rho_{\mathfrak{n}} -
\left(\sum_{\alpha \in \Delta(\mathfrak{n}) \cap \Delta'}
															\alpha - \sum_{\beta' \in \left(\Delta(\fn)\setminus S \right) \cap \Delta'}
															\beta' + \sum_{\beta \in S\cap (-\Delta')}
															\beta\right) \\
=&\ \tau' + \sum_{\gamma \in \Delta_M,\ m_{\gamma} \geq 0} m_{\gamma}\gamma + \sum_{\beta' \in S_1}
															\beta' - \sum_{\beta \in S_2}
															\beta,
\end{aligned}
\end{equation}
where $S_1:=\left(\Delta(\fn)\setminus S \right)\cap \Delta'$ and $S_2:=S\cap (-\Delta')$.

Consider the squared norm of the weight in \eqref{otherktypes}:
\begin{equation} \label{norm2}
\begin{aligned}
&\left\|\tau' + \sum_{\gamma \in \Delta_M,\ m_{\gamma} \geq 0} m_{\gamma}\gamma + \sum_{\beta' \in S_1}
															\beta' - \sum_{\beta \in S_2}
															\beta\right\|^2=\ ||\tau'||^2 +\\
& 2\left\langle \tau', \sum_{\gamma \in \Delta_M,\ m_{\gamma} \geq 0} m_{\gamma}\gamma + \sum_{\beta' \in S_1}
															\beta' - \sum_{\beta \in S_2}
															\beta \right\rangle + \left\|\sum_{\gamma \in \Delta_M,\ m_{\gamma} \geq 0} m_{\gamma}\gamma + \sum_{\beta' \in S_1}
															\beta' - \sum_{\beta \in S_2}
															\beta\right\|^2
\end{aligned}
\end{equation}
By construction, $\tau'$ is a dominant weight in $\Delta'$. On the other hand, we have seen from above that
$$\gamma \in \Delta_M \subset \Delta';\ \ \ \ \beta' \in \Delta';\ \ \ -\beta \in \Delta'.$$
 Thus $\langle \tau', \gamma \rangle$, $\langle \tau', \beta' \rangle$, $\langle \tau', -\beta \rangle$
are all non-negative. Therefore,
$$
\left\|\tau' + \sum_{\gamma \in \Delta_M,
\ m_{\gamma} \geq 0} m_{\gamma}\gamma + \sum_{\beta' \in S_1}	\beta' - \sum_{\beta \in S_2} \beta\right\|^2\geq \|\tau'\|^2.
$$
Equality occurs exactly when $\delta_M' = \delta_M$, and  $S_1$, $S_2$ are both empty. The latter condition further implies that $S=\Delta(\mathfrak{n}) \cap \Delta'$.

Since $V_{\mathfrak{k}}(\tau')|_M$ has $\widetilde{K \cap M}-$types of norm less than or equal to $\tau'$, the left module in \eqref{eq:mult} contains
$V_{\mathfrak{k} \cap \mathfrak{m}}(\tau')$ with multiplicity equal to $[\pi_\fk m\otimes V_{\mathfrak{k} \cap \mathfrak{m}}(\rho_\fk
m): V_{\mathfrak{k} \cap \mathfrak{m}}(\delta_M)]$.
\end{proof}

\smallskip
We now present the proof of Theorem \ref{thm:main1} and Theorem \ref{thm:main} for all $\pi = {\rm Ind}_{M}^G(\mathbb{C}_{\xi} \otimes \pi_u)$
in Theorem \ref{thm:unitarydual}. The same argument holds for Spin groups with
$\pi_u = triv$. It suffices to prove
\begin{equation} \label{eq-mult10}
[\pi \otimes V_{\mathfrak{k}}(\rho) : V_{\mathfrak{k}}(2\la - \rho)]
= \begin{cases} 1 & \text{if } \pi_u \in \widehat{M}^d \\
0 &\text{if } \pi_u \notin \widehat{M}^d
\end{cases}.
\end{equation}

The special case when $M = G$ and $\pi = \pi_u$ is the content of  Section 7.

By applying $\pi_\fk m = \pi_u \in \widehat{M}^d$ to \eqref{eq-mult} and \eqref{eq:mult},
$$\left[\pi \otimes V_{\mathfrak{k}}(\rho) : V_{\mathfrak{k}}(\tau')\right] = \left[\pi_u \otimes V_{\mathfrak{k} \cap \mathfrak{m}}(\rho_\fk m): V_{\mathfrak{k} \cap \mathfrak{m}}(\delta_M)\right].$$
When $\pi_u \in \widehat{M}^d$, the proof in Proposition \ref{prop-positive-result} implies that
$\pi_u$ has a unique spin-lowest $K-$type and hence the right hand side of the above equation is equal to $1$.

\medskip
The case $H_D(\pi_u) = 0$ occurs in Types C and D only. By the proof
of Proposition \ref{prop-branching-Dirac}, in particular Equation \eqref{otherktypes},
the $\widetilde{K\cap M}$-types appearing in the left module on the last line of \eqref{eq-mult} has highest weights
\begin{equation} \label{eq-parity}
\tau' + \sum_{\gamma \in \Delta_M,\ m_{\gamma} \geq 0} m_{\gamma}\gamma + \sum_{\beta' \in S_1}
															\beta' - \sum_{\beta \in S_2}
															\beta + {\bf e_i},
\end{equation}														
where ${\bf e_i}$ is the unit vector corresponding to the bolded ${\bf 1}$
in the proof of Proposition \ref{prop-positive-result}.
Consider the sum of coordinates of the expression in \eqref{eq-parity}:
since all the roots are of the form $2{\bf e_i}$ and/or ${\bf e_i \pm e_j}$
in Type C and D, the sum of coordinates in \eqref{eq-parity} must be of
opposite parity with that of $\tau'$.
Therefore, the multiplicity $\left[\pi \otimes V_{\mathfrak{k}}(\rho) : V_{\mathfrak{k}}(\tau')\right]$ in \eqref{eq-mult} is zero.

Hence \eqref{eq-mult10} holds, and this completes the proofs of
Theorems \ref{thm:main1} and \ref{thm:main}.\qed

\appendix

\section{The notion of unipotent representation}
James Arthur made conjectures in the 1980's which state (roughly) that
automorphic representations occurring in the residual spectrum of a
locally symmetric space associated to a number field $F$, should
be associated to $^\vee G-$equivalence classes of homomorphisms
$$
\Phi:\C W_F\times SL(2)\longrightarrow\ ^\vee G
$$
where $\C W_F$ is the Weil group. There are additional conditions such
as the image not contained in any proper
Levi component, and $\Phi(\C W_F)$ be bounded.
We refer to \cite{A} for a  very detailed analysis.
For $F$ a local field,  one
expects such representations to be the building blocks of the unitary
dual.
The homomorphism $\Phi\mid _{\bb   C^\times}$ determines a
  semisimple orbit and, in the case of $F=\bb C$ (which is the case in
  this paper)  should correspond to unitary induction. The
  infinitesimal character conjectured by Arthur is
  $$
  d\Phi\left(1,\begin{pmatrix} 1/2&0\\0&-1/2\end{pmatrix}\right).
  $$
When $\Phi\mid_{\C W_F}=Triv$, the infinitesimal character is $^\vee h/2$
  where  $\{ ^\vee e,\ ^\vee h,\ ^\vee f\}$ is a Lie triple associated to $\Phi(SL(2)).$
In the general case, the data for $\Phi$ correspond to a $^\vee G-$orbit, semisimple times
unipotent.

%
%

 In \cite{BV}, for the above reason, the special case  $\Phi|_{\bb
    C^\times}=Triv$ is studied.  These correspond to unipotent
  conjugacy classes. A set of representations $\pi$ associated to
    $\Phi$ are assumed to satisfy
  \begin{itemize}
  \item $\Ann(\pi) \subset U(\fk g)$ is maximal subject to the
    prescribed infinitesimal character.
  \end{itemize}
These representations are called {\it special unipotent Arthur
packets} associated to the nilpotent orbit in $^\vee \fk g$ determined by
$\Phi.$  The main result is that these packets satisfy the properties
conjectured by Arthur.

The {\it building blocks} of the unitary dual is conjectured to be the packets
associated to $\Phi$ satisfying $\Phi\mid_{\C W_F}=Triv$ and such
that the orbit of $^\vee e$ does not
meet any proper Levi component. It is clear that this cannot be the
case; the best known example is $G=Sp(2n,\bb C)$ and the
Segal-Shale-Weil (also called oscillator) representation. It is
unitary, not unitarily induced from any representation on a proper
Levi component, and its infinitesimal character is not of the form
$^\vee h/2.$

For $GL(n,\bb C)$, the unitary dual is determined in \cite{V1}, and
for the other classical groups in  \cite{B1}. The building blocks for
$GL(n,\bb C)$ are 1-dimensional unitary
representations of Levi components. For the other groups, a set of building blocks is
identified explicitly in \cite{B1}. They can be characterized as
irreducible representations which are
  \begin{itemize}
  \item unitary with half-integral infinitesimal character,
    \item their annihilator in the universal
  algebra is maximal for the given infinitesimal character.
\end{itemize}
They have properties analogous to the Arthur packets of special
unipotent representations. A minimal set of building blocks
requires that the representations not be unitarily induced irreducible
from proper Levi components. In \cite{B1}  the class of unipotent
representations is extended to include some unitarily induced representations from
proper Levi factors (and even some representations in complementary
series which fall under the category of special unipotent).
This is in line with the parameters introduced by
Arthur where the image of $\Phi$ meets a proper Levi component.
A parametrization in terms of the homomorphism $\Phi$ is given in
\cite[Chapter 11]{BV}; the infinitesimal
character is modified according to certain elements in the centralizer of the Lie
triple.

\medskip
A different parametrization, motivated by  the \textit{orbit philosophy}
is in \cite{B3}. It is in terms of nilpotent orbits $\CO\subset \fk g.$
It is shown there that they can be  obtained by iterating
  $\Theta-$lifts and tensoring with unitary characters starting with a
  1-dimensional representation on $O(n,\bb C)$ or the trivial of $Sp(2n,\bb C).$

Another definition of unipotent representations  is given and studied in
\cite{LMM}.  It is our understanding that the representations
listed  below match those in \cite{LMM}.

The packets associated to $^\vee h/2$ are called \textit{special
  unipotent}. For the more general
infinitesimal characters, they are called  \textit{unipotent}.
To be completely clear what we mean by unipotent representation, the
list of infinitesimal characters is in the next section.

\subsection{Parameters of Unipotent Representations}
\label{sec:upar}
{
We rely on \cite{BV} and \cite{B3}. For each $\CO\subset\fk g$ we
will give an infinitesimal character $(\la_\CO,\la_\CO)$, and a set of
$(\la_\CO,w\la_\CO)$ such that  $\{L(\la_{\C O},w\la_{\C O})\}$ are
the unipotent representations with asymptotic support $\CO.$
In all cases $\la_{\C O}$ and $-\la_{\C O}$ are in the same $W-$orbit.

\bigskip
\noindent\textbf{Main Properties of }$\boldsymbol{\la_{\CO}}$.
Suppose $\Pi$ is an
irreducible representation with infinitesimal character
$(\la_\CO,\la_\CO).$ Then $\la_\CO$ and $\Pi$ must satisfy:

\begin{enumerate}
\item $\Ann(\Pi)\subset U(\fk g)$ is the maximal primitive ideal
  $\C I_{\la_\CO}$ with infinitesimal character $(\la_{\C O},\la_\CO)$,

\item {$|\{\Pi\ :\ \Ann(\Pi)=\C I_{\la_\CO}\}| =\mid \widehat{A(\C
    O)}\mid,$ where $A(\CO)$ is the component group of the centralizer
  of an $e\in \C O,$}

\item $\Pi$ is unitary.
\end{enumerate}
We call such representations unipotent. The list of $\la_\CO$ is given below.  The choices satisfying
(3) rely on the determination of the unitary dual for classical groups
in \cite{B1}. The parameter will always have integer and half-integer coordinates, and the
corresponding system of integral co-roots is maximal.}

\begin{definition}
A special orbit $\CO$ (in the sense of Lusztig)
is called \textbf{stably trivial} if Lusztig's
quotient $\ovl{A}(\CO)$ equals the full component group $A(\CO).$
\end{definition}
For a definition and discussion of $\ovl{A}(\CO),$ see \cite{L},
chapter 13.

\smallskip
The set of unipotent representations as defined above contains the
\textrm{building blocks} of the unitary dual. They are attached  to
$\CO$ which are not induced (in the sense of Lusztig-Spaltenstein) from
any proper Levi component. For $\CO$ special (in the sense of Lusztig)
and not induced from a nilpotent orbit on a proper Levi component,
$\la_{\CO} = h(^\vee\CO)/2$ where $^\vee\CO$ is the
Barbasch-Spaltenstein-Vogan dual of $\CO.$ For other
special $\CO$ which are induced from proper Levi components, condition (2) may not be satisfied
if they are not stably trivial. See the example below. The component
group $A(\C O)$ depends on the isogeny class of $G.$ To make a
definition that includes all cases, one would have to take the
isogeny class into account. We leave this for  future
considerations. It is our understanding that a  definition of unipotent closely related to the one
above is considered in \cite{LMM} addresses this problem.

\medskip
{The partitions in the next examples denote rows}.
\begin{example}\

$\bullet$\ $\CO=(2222)\subset \mathfrak{sp}(8)$ is stably
  trivial, $A(\CO)=\ovl{A(\CO)}\cong\bZ_2$, $\la_\CO=(2,1,1,0).$
In this case $^\vee \C O$ corresponds to the partition $(531),$ and
$\la_\CO=h(^\vee \C O)/2.$

$\bullet$\ $\CO=(222)\subset \mathfrak{sp}(6)$ has dual orbit $^\vee \C O$ corresponding to $(331)$
but is not stably trivial; $A(\CO)\cong\bZ_2,$ while $\ovl{A(\C O)}\cong 1$.
In this case $h(^\vee \C O)/2=(1,1,0),$ and for
this infinitesimal character, conditions (1) and (3) are satisfied,
but (2) is not satisfied. The choice of infinitesimal character in
this case will be $\la_\CO=(3/2,1/2,1/2).$ There are two parameters,
$$
\be{pmatrix}
\la_L\\ \la_R
\ee{pmatrix}=
\be{pmatrix}
3/2&1/2&1/2\\
3/2&1/2&1/2
\ee{pmatrix}\qquad\text{ and }\qquad
\be{pmatrix}
3/2&1/2&1/2\\
1/2&3/2&-1/2
\ee{pmatrix}
$$
Note that $(1,1,0)$ is in the root lattice and drops down to the
adjoint group, $(3/2,1/2,1/2)$ while is not, so genuine for $Sp(2n,\bb C)$.

$\bullet$\ $\CO=(211)$ in $\mathfrak{sp}(4,\bb C)$ is not special in the sense of
Lusztig. The parameter is $\la_{\CO}=(3/2,1/2)$ and the
  representations are the two components of the oscillator
  representation:
  $$
\be{pmatrix}
\la_L\\ \la_R
\ee{pmatrix}=
\be{pmatrix}
3/2&1/2\\
3/2&1/2
\ee{pmatrix}\qquad\text{ and }\qquad
\be{pmatrix}
3/2&1/2\\
3/2&-1/2
\ee{pmatrix}
$$
\end{example}

\subsection{Type A}\label{sec:a} The group $G$ is $GL(n).$
Nilpotent orbits are determined by their
Jordan canonical form. An orbit is given by a partition, \ie a
sequence of numbers in decreasing order $\CO\longleftrightarrow (n_1,\dots
,n_k)$ that add up to $n.$  Let $(m_1,\dots ,m_l)$ be the dual
partition. The component group of $\C O$ is trivial.
The infinitesimal character is
$$
\la_\CO=\bigg(\frac{m_1-1}{2},\dots ,-\frac{m_1-1}{2},\dots
,\frac{m_l-1}{2},\dots, -\frac{m_l-1}{2}\bigg).
$$
The orbit is induced from the trivial orbit on the Levi component $\fm$
of a parabolic subalgebra  $\fp=\fm +\fn$ with
$\fm= \mathfrak{gl}(m_1)\times \dots \times \mathfrak{gl}(m_l).$ The corresponding
unipotent representation is spherical and induced irreducible from the
trivial representation on the same Levi component. {All orbits are
{\it special } and {\it stably trivial}}.
\subsection{Type B}\label{sec:b}
{
We describe the case $SO(2m+1)$. For $O(2m+1)$ there are twice the
parameters, the parameters for $SO$  are tensored with $sgn$.

\medskip
A nilpotent orbit is determined by its
Jordan canonical form (in the standard representation). Then $\CO$  is
parametrized by a partition $\CO\longleftrightarrow (n_1,\dots ,n_k)$
of $2m+1$ such that every even entry occurs an even number of times.
Let $(m'_0,\dots ,m'_{2p'})$ be the transpose partition (add an $m'_{2p'}=0$ if
necessary, in order to have an odd number of terms).  If $\CO$ is
represented by a tableau, these are the sizes of the columns in
decreasing order.
If there are any $m'_{2j}=m'_{2j+1}$, then pair them
together and remove them from the partition.
Then relabel and pair up the remaining columns $(m_0)(m_1,m_2)\dots
(m_{2p-1}m_{2p}).$ The members of each
pair have the same parity and $m_0$ is odd. $\la_\CO$ is given by the
coordinates
\begin{equation}
  \label{eq:unipb}
\aligned
(m_0)&\longleftrightarrow (\frac{m_0-2}{2},\dots ,\frac12),\\
(m'_{2j}=m'_{2j+1})&\longleftrightarrow (\frac{m_{2j}'-1}{2},\dots ,
-\frac{m_{2j}'-1}{2})\\
(m_{2i-1}m_{2i})&\longleftrightarrow (\frac{m_{2i-1}}{2},\dots , -\frac{m_{2i}-2}{2}).
\endaligned
\end{equation}
}

In case $m'_{2j}=m'_{2j+1},$ $\CO$ is induced from an orbit
$$
\CO_\fk m\subset\fm= \mathfrak{so}(*)\times \mathfrak{gl}\big(\frac{m'_{2j}+m'_{2j+1}}{2}\big)
$$
where $\fk m$ is the
Levi component of a parabolic subalgebra $\fp=\fm +\fn$. $\CO_\fk m$ is
the trivial nilpotent on the $\mathfrak{gl}-$factor. The
component groups satisfy $A_G(\CO)\cong A_M(\CO_{\fk m}).$
Each unipotent representation is
unitarily induced from a unipotent representation attached to
$\CO_\fk m.$

Similarly if some $m_{2i-1}=m_{2i},$ then $\CO$ is
induced from a
$$
\CO_{\fk m}\subset \fk m\cong \mathfrak{so}(*)\times
\mathfrak{gl}(\frac{m_{2i-1}+m_{2i}}{2}),\quad  (0)\ \text{ on the }
  \mathfrak{gl}-\text{factor.}
$$
Here $A_G(\CO)\not\cong A_M(\CO_{\fk m}),$ but
each unipotent representation is (not necessarily unitarily) induced
irreducible from a representation on the Levi component
$\fk m$, unipotent on
$\mathfrak{so}(*)$, and a character on the $\mathfrak{gl}$-factor.

\medskip

{The {\it stably trivial} orbits are the
ones such that every odd sized part appears an even number of
times, except for the largest size}. An orbit is called triangular if it has
partition
$$
\CO\longleftrightarrow(2m+1,2m-1,2m-1,\dots,3,3,1,1).
$$


We give the explicit Langlands parameters of the unipotent
representations. There are $\mid A_G(\CO)|$ distinct representations.
Let
$$
(\underset{r_k}{\underbrace{k,\dots ,k}},\dots
,\underset{r_1}{\underbrace{1,\dots 1}})
$$
be the rows of the Jordan form of the
nilpotent orbit. The numbers $r_{2i}$ are even.
The reductive part of the centralizer (when $G$ is the orthogonal
group) of the nilpotent element is a product of $O(r_{2i+1})$, and
$Sp(r_{2j})$.

The columns are paired as in (\ref{eq:unipb}). The pairs
$(m'_{2j}=m'_{2j+1})$ contribute to the spherical part of the
parameter,
\begin{equation}
  \label{eq:unipbprime}
(m'_{2j}=m'_{2j+1})\longleftrightarrow
\begin{pmatrix}
  \la_L\\ \la_R
\end{pmatrix}
=
\begin{pmatrix}\frac{m'_{2j}-1}{2}&,&\dots&,&
&-\frac{m'_{2j}-1}{2}\\
\frac{m'_{2j}-1}{2}&,&\dots&,& &-\frac{m'_{2j}-1}{2}
\end{pmatrix}  .
\end{equation}
The singleton $(m_0)$ contributes to the spherical part,
\begin{equation}
\label{eq:unipbzero}
(m_0)\longleftrightarrow
 \begin{pmatrix}
\frac{m_{0}-2}{2}&,&\dots&,&\frac{1}{2}\\
\frac{m_{0}-2}{2}&,&\dots&,&\frac{1}{2}
\end{pmatrix}  .
\end{equation}
Let $(\eta_1,\dots ,\eta_{p})$ with $\eta_i=\pm 1,$ one for each
$(m_{2i-1},m_{2i})$. An $\eta_i=1$ contributes to the spherical part
of the parameter, with coordinates as in (\ref{eq:unipbprime}) and (\ref{eq:unipbzero}). An
$\eta_i=-1$ contributes
\begin{equation}
\label{eq:unipbep}
\begin{pmatrix}
\frac{m_{2i-1}}{2}&,&\dots&,&\frac{m_{2i}+2}{2}&\frac{m_{2i}}{2}&,&\dots
&,&-\frac{m_{2i}-2}{2}\\
\frac{m_{2i-1}}{2}&,&\dots&,&\frac{m_{2i}+2}{2}&\frac{m_{2i}-2}{2}&,&\dots
&,&-\frac{m_{2i}}{2}
\end{pmatrix}.
\end{equation}
{If $m_{2p}=0,$ $\eta_p=1$ only for $SO$.}
\subsection*{Explanation}
  \begin{enumerate}
  \item Odd sized rows contribute a $\bZ_2$ to $A(\CO),$
    even sized rows a $1.$
\item When there are no $m'_{2j}=m'_{2j+1},$ every row size occurs.
{The inequalities
$$\dots (m_{2i-1}\ge m_{2i})>(m_{2i+1}\ge  m_{2i+2})\dots
$$}
imply that there are $m_{2i}-m_{2i+1}$ rows of size $2i+1.$ Each
pair $(m_{2i-1}\ge m_{2i})$ contributes exactly 2 parameters
corresponding to the $\bZ_2$ in $A(\CO)$.
\item The pairs $(m'_{2j}=m'_{2j+1})$ lengthen the sizes of the rows
  without changing their parity. The component group does not change,
  they do not affect the  number of parameters.
  \end{enumerate}
As already mentioned, when $G=O(2m+1,\bC)$ the unipotent
representations are obtained from those of $SO(2m,\bC)$ by lifting
them to $O(2m,\bC)$, and also tensoring with $sgn$.

\bigskip
{
  In case $m_{2i-1}=m_{2i}$ even, there is another choice of parameter:
\begin{equation}\label{B:sunipotent}
(m_{2i-1}=m_{2i})\longleftrightarrow (\frac{m_{2i-1}-1}{2},\dots ,-\frac{m_{2i}-1}{2}).
\end{equation}
The representations are unitarily induced irreducible from
representations of the same type on  Levi
components $GL(2m_{2i-1})\times SO(2n+1-2m_{2i-1}).$   The number of
parameters no longer matches $|A(\CO)|,$ but special unipotent
representations are included.
}

\subsection{Type C}\label{sec:c}

{
A nilpotent orbit is determined by its
Jordan canonical form (in the standard representation). It is
parametrized by a partition $\CO\longleftrightarrow(n_1,\dots ,n_k)$
of $2n$ such that
every odd part occurs an even number of times.
Let $(c'_0,\dots ,c'_{2p'})$ be the dual partition (add a $c'_{2p'}= 0$ if
necessary in order to have an odd number of terms). As in type B,
these are the sizes of the columns of the tableau corresponding to $\CO$.
If there are any $c'_{2j-1}=c'_{2j}$ pair them up and set aside. Then relabel and pair
up the remaining columns $(c_0c_1)\dots
(c_{2p-2}c_{2p-1})(c_{2p}).$ The members of each pair have the same
parity. The last one, $c_{2p},$ is always even. Then
form a parameter
\begin{align}
(c'_{2j-1}=c'_{2j})&\longleftrightarrow (\frac{c_{2j}-1}{2},\dots ,
-\frac{c_{2j}-1}{2}),\label{eq:c1}\\
(c_{2i}c_{2i+1})&\longleftrightarrow (\frac{c_{2i}}{2},\dots ,
-\frac{c_{2i+1}-2}{2}),\label{eq:c2}\\
c_{2p}&\longleftrightarrow (\frac{c_{2p}}{2},\dots , 1).\label{eq:c3}
\end{align}
}

The nilpotent orbits and the unipotent representations have the same
properties with respect to these pairs as the corresponding ones in
type B.

{The {\it stably trivial} orbits are the
ones such that every even sized part appears an even number of
times}.

An orbit is called triangular if it corresponds to the
partition $(2m,2m,\dots,4,4,2,2).$


\bigskip
We give a parametrization of the unipotent representations in terms of
their Langlands parameters. There are $\mid A_G(\CO)\mid$
representations.

Let
$$
(\underset{r_k}{\underbrace{k,\dots ,k}},\dots
,\underset{r_1}{\underbrace{1,\dots ,1}})
$$
be the rows of the Jordan form of the
nilpotent orbit. The numbers $r_{2i+1}$ are even.
{
The reductive part of the centralizer of the
nilpotent element is a product of $Sp(r_{2i+1})$, and $O(r_{2j})$.

The elements $(c'_{2j-1}=c'_{2j})$ and $c_{2p}$ contribute to the
spherical part of the parameter as in (\ref{eq:unipbprime}) and
(\ref{eq:unipbzero}).
Let $(\eta_1,\dots ,\eta_p)$ be such that $\eta_i=\pm 1,$ one for each
$(c_{2i},c_{2i+1}).$  An $\eta_i=1$ contributes to the spherical
part, according to the infinitesimal character. An $\eta_i=-1$ contributes
\begin{equation}
\label{eq:unipcep}
\begin{pmatrix}
&\frac{c_{2i}}{2}&,&\dots&,&\frac{c_{2i+1}+2}{2}&\frac{c_{2i+1}}{2}&\dots
&,&-\frac{c_{2i+1}-2}{2}\\
&\frac{c_{2i}}{2}&,&\dots&,&\frac{c_{2i+1}+2}{2}&\frac{c_{2i+1}-2}{2}&\dots &,&-\frac{c_{2i+1}}{2}\\
\end{pmatrix}.
\end{equation}
}
The explanation is similar to type B.

\bigskip
{
  In case $c_{2i}=c_{2i+1}$ odd, there is another choice of parameter:
\begin{equation}\label{C:sunipotent}
(c_{2i}=c_{2i+1})\longleftrightarrow (\frac{c_{2i-1}-1}{2},\dots ,-\frac{c_{2i}-1}{2}).
\end{equation}
The representations are unitarily induced irreducible from
representations of the same type on  Levi
components \newline $GL(2c_{2i}+1)\times Sp(2n-2c_{2i}).$   The number of
parameters no longer matches $|A(\CO)|,$ but special unipotent
representations are included.
}
\subsection{Type D}\label{sec:d}

{
We treat the case $G=SO(2m).$
A nilpotent orbit is determined by its
Jordan canonical form (in the standard representation). It is
parametrized by a partition $\CO\longleftrightarrow (n_1,\dots ,n_k)$
of $2m$ such that
every even part occurs an even number of times.
Let $(m'_0,\dots ,m'_{2p'-1})$ be the dual partition (add a $m'_{2p'-1}=0$ if
necessary), the sizes of the columns of the tableau corresponding to
$\CO.$  If there are any $m'_{2j}=m'_{2j+1}$ pair them up and
remove from the partition.
Then pair up the remaining columns $(m_0,m_{2p-1})(m_1,m_2)\dots
(m_{2p-3},m_{2p-2}).$ The members of each pair have the same parity and
$m_0, m_{2p-1}$ are both even. The infinitesimal character is
\begin{equation}
  \label{eq:unipd}
\aligned
(m'_{2j}=m'_{2j+1})&\longleftrightarrow (\frac{m'_{2j}-1}{2}\dots , -\frac{m'_{2j}-1}{2})\\
(m_0m_{2p-1})&\longleftrightarrow (\frac{m_0-2}{2},\dots ,-\frac{m_{2p-1}}{2}),\\
(m_{2i-1}m_{2i})&\longleftrightarrow (\frac{m_{2i-1}}{2}\dots , -\frac{m_{2i}-2}{2})
\endaligned
\end{equation}
}
{
The nilpotent orbits and the unipotent representations have the same
properties with respect to these pairs as the corresponding ones in
type B.
An exception occurs for $G=SO(2m)$ when the partition is formed of pairs
$(m'_{2j}=m'_{2j+1})$ only. In this case there are two nilpotent
orbits corresponding to the partition. There are also two nonconjugate
Levi components of the form $\mathfrak{gl}(m'_0)\times \mathfrak{gl}(m'_2)\times \dots
\mathfrak{gl}(m'_{2p'-2})$ of parabolic subalgebras. There are two unipotent
representations each induced irreducible from the trivial
representation on the corresponding Levi component.

{The {\it stably
trivial} orbits are the ones such that every even sized part appears
an even number of times}.

A nilpotent orbit is triangular if it
corresponds to the partition $(2m-1,2m-1,\dots ,3,3,1,1).$

}
{
The parametrization of the unipotent representations follows from types
B,C, with the pairs $(m'_{2j}=m'_{2j+1})$ and $(m_0,m_{2p-1})$
contributing to the spherical part of the parameter only. Similarly for
$(m_{2i-1},m_{2i})$ with $\ep_i=1$ spherical only, while $\ep_i=-1$
contributes analogous to (\ref{eq:unipbep}) and (\ref{eq:unipcep}).

The explanation parallels that for types B, C.

\bigskip
When $G=O(2m,\bC)$ the unipotent representations are obtained from
those of $SO(2m,\bC)$ by lifting them to $O(2m,\bC)$, and also
tensoring with $sgn$. In the case when all $m'_{2j}=m'_{2j+1}$ the
representations associated to the two nilpotent orbits have the same
lift, and it is invariant under tensoring with $sgn$. Otherwise
tensoring with $sgn$ gives inequivalent unipotent representations.
}

\medskip
As in types B,C, when $m_{2i-1}=m_{2i}$ is even, there is another choice of infinitesimal character:
\begin{equation}\label{D:sunipotent}
(m_{2i-1}=m_{2i})\longleftrightarrow (\frac{m_{2i-1}-1}{2},\dots ,-\frac{m_{2i}-1}{2}).
\end{equation}
The representations are unitarily induced irreducible from
representations of the same type on  Levi
components $GL(2m_{2i})\times SO(2n-2m_{2i-1}).$   The number of
parameters no longer matches $|A(\CO)|,$ but special unipotent
representations are included.

\section{Some Atlas Calculations}
In this section, we illustrate some of the results on  signatures on
cx-relevant $K-$types considered in Sections 4--6 using the software
\texttt{atlas} \cite{ALTV, At}.
The calculations are carried out using
the function \texttt{print\_sig\_irr\_long}, which is available at
\medskip
\begin{center}
\texttt{http://klein.mit.edu/$\scriptstyle\sim$dav/atlassem/bottom.at}.
\end{center}
\bigskip

\subsection{Section \ref{sec-bn>k}, Equation \eqref{eq:brelevant}} Let $G = SO(7,\mathbb{C})$,
and $\la = (-1/2; -2,-1)$.
The \texttt{atlas} is
\begin{verbatim}
atlas> set G = complexification(SO(7))
atlas> set all = all_parameters_gamma(G,[4,2,1,4,2,1]/2)
atlas> all[0]
Value: final parameter(x=47,lambda=[5,3,1,5,3,1]/2,nu=[4,2,1,4,2,1]/2)
\end{verbatim}

The signature of some of the $K-$types are given by:
\begin{verbatim}
atlas> print_sig_irr_long(all[0],KGB(G,0),15)
sig  x  lambda                        hw                          dim
s    0  [  1,  1,  1, -1, -1, -1 ]/2  [ -2, -1,  0,  2,  1,  0 ]  1
s    0  [  1,  1,  1,  1,  1, -1 ]/2  [ -2, -1,  0,  3,  2,  0 ]  21
1    0  [ 1, 1, 1, 1, 1, 1 ]/2        [ -2, -1,  0,  3,  2,  1 ]  35
\end{verbatim}
The $K-$types of $J(\la,-s\la)$ are in
the column labelled \texttt{hw}. More precisely, by adding the $i^{th}$-coordinate
and the $(i + \text{rank}(G))^{th}$-coordinate of the vector in the \texttt{hw}
column, one can get the highest weight of a $K-$type in usual coordinates.
For example, $[ -2, -1,  0,  3,  2,  0 ]$ corresponds to the highest
weight $(-2+3, -1+2, 0+0) = (1,1,0)$ in the usual coordinates.

The \texttt{sig} column represents the signature of the Hermitian form of $J(\la_{rel},-s\la_{rel})$.
The form is definite if and only if the entries of the \texttt{sig} column are all scalars or all scalar multiples of \texttt{s}. In particular,
the above output shows that the form is indefinite on the $K-$types $V_{\mathfrak{k}}(1,1,0)$
and $V_{\mathfrak{k}}(1,1,1)$, which matches Equation \eqref{eq:brelevant}.

\bigskip

\subsection{Section \ref{sec:additionalb}, Case (c)} Let $G = SO(9,\mathbb{C})$ and
$\la=(-5/2, -3/2, -1/2) \cup (2)$. We are in the setting of Case (c). Its $K-$type signatures are given by
\begin{verbatim}
sig x  lambda                      hw                        dim
1   0  [ 1, 1, 1, 1,-1,-1,-1,-1]/2 [-3,-2,-1, 0, 3, 2, 1, 0] 1
1   0  [ 1, 1, 1, 1, 1, 1,-1,-1]/2 [-3,-2,-1, 0, 4, 3, 1, 0] 36
s   0  [ 3, 1, 1, 1, 1,-1,-1,-1]/2 [-2,-2,-1, 0, 4, 2, 1, 0] 44
1   0  [ 3, 3, 1, 1, 1, 1,-1,-1]/2 [-2,-1,-1, 0, 4, 3, 1, 0] 495
s   0  [ 3, 1, 1, 1, 3, 1,-1,-1]/2 [-2,-2,-1, 0, 5, 3, 1, 0] 910
\end{verbatim}
In this case, the $K-$types $V_{\mathfrak{k}}(1,1,0,0)$ and $V_{\mathfrak{k}}(2,0,0,0)$ have different signatures.

\bigskip

\subsection{Section \ref{sec:nonshc}, non-spherical Type C}
Let $G = Sp(8,\mathbb{C})$ and parameter
$\begin{pmatrix} 1/2 \\ -1/2  \end{pmatrix} \cup (-2,-1) \cup (3/2)$.
The \texttt{atlas} code for this parameter is
\begin{verbatim}
atlas> set G = Sp(8,C)
atlas> set all = all_parameters_gamma(G,[4,3,2,1,4,3,2,1]/2)
atlas> LKT(all[1])
Value: (KGB element #0,[ 1, 0, 0, 0, 0, 0, 0, 0 ]/1)
\end{verbatim}

The signatures of the $K-$types are:
\begin{verbatim}
sig  x  lambda                  hw                          dim
1    0  [ 1,0,0,0,0,0,0,0 ]/1  [ -3,-3,-2,-1, 4, 3, 2, 1 ]  8
s    0  [ 1,1,1,0,0,0,0,0 ]/1  [ -3,-2,-1,-1, 4, 3, 2, 1 ]  48
\end{verbatim}
The $K-$types $V_{\mathfrak{k}}(1,0,0,0)$ and $V_{\mathfrak{k}}(1,1,1,0)$ have different signatures.

\bigskip

\subsection{Section \ref{sec-bn>k1}, Equation \eqref{eq:factor}}  This is an example where the Hermitian
  form is indefinite on a single $K-$type. Let $G = SO(6,\mathbb{C})$
  and the parameter be given  by $(-3/2, -1/2;0)$. Then the signatures are given by:
\begin{verbatim}
sig  x  lambda                        hw                          dim
1    0  [ 0, 0, 0, 0, 0, 0 ]/1        [ -2, -1,  0,  2,  1,  0 ]  1
1+s  0  [ 1, 1, 0, 0, 0, 0 ]/1        [ -1,  0,  0,  2,  1,  0 ]  15
1    0  [ 1, 0, 0, 1, 0, 0 ]/1        [ -1, -1,  0,  3,  1,  0 ]  20
s    0  [ 1, 1, 1, 1, 0, 0 ]/1        [ -1,  0,  1,  3,  1,  0 ]  45
\end{verbatim}
The $K-$type $V_{\mathfrak{k}}(1,1,0)$ has indefinite signature as in Equation
\eqref{eq:factor} with an odd number of spherical coordinates.

\bigskip

\subsection{Section \ref{sec:nsphd}, non-spherical Type D}  Let $G=SO(10,\mathbb{C})$.
Let $\begin{pmatrix} 1/2 \\ -1/2 \end{pmatrix} \cup (-2,-1,0) \cup (5/2)$
be the parameter, where the spherical part satisfies Case (c) of
Section \ref{sec:dother}. Then the signatures of the $K-$types are given by:
\begin{verbatim}
sig   x  lambda                    hw                          dim
s     0  [1,0,0,0,0,0,0,0,0,0 ]/1  [-3,-3,-2,-1,0,4,3,2,1,0 ]  10
s     0  [1,1,1,0,0,0,0,0,0,0 ]/1  [-3,-2,-1,-1,0,4,3,2,1,0 ]  120
1+2s  0  [1,1,0,0,0,1,0,0,0,0 ]/1  [-3,-2,-2,-1,0,5,3,2,1,0 ]  320
1+s   0  [2,0,0,0,0,1,0,0,0,0 ]/1  [-2,-3,-2,-1,0,5,3,2,1,0 ]  210
s     0  [1,1,1,1,0,1,0,0,0,0 ]/1  [-3,-2,-1,0,0,5,3,2,1,0 ]   1728
\end{verbatim}
The $K-$types $V_{\mathfrak{k}}(1,1,1,0,0)$ and $V_{\mathfrak{k}}(2,1,0,0,0)$ have opposite signatures.
Moreover, this is the only place where the signatures are different on the
level of cx-relevant $K-$types.

\ifx\undefined\bysame
\newcommand{\bysame}{\leavevmode\hbox to3em{\hrulefill}\,}
\fi

\bigskip

\centerline{\scshape Funding}
Barbasch is supported by NSF grant 2000254. Dong is supported by the
National Natural Science Foundation of China (grant 12171344). Wong is supported by the Presidential Fund of
CUHK(Shenzhen) and the National Natural Science Foundation of China
(grant 11901491).

\bigskip
\centerline{\scshape Acknowledgements}
This work was initiated at the Workshop on Lie Group Representations and Automorphic Forms
at Sichuan University in October 2018. The authors would like to thank
the institution for their hospitality. The authors would also like to
thank an anonymous referee for his/her suggestions.

%
%
%
%
%
%
%


\begin{thebibliography}{ALTV}

\bibitem[A]{A}
J.~Arthur, {\it The endoscopic classification of
  representations. Orthogonal and symplectic groups}, American
Mathematical Society Colloquium Publications, 61 American Mathematical
Society, Providence, RI, 2013

\bibitem[ALTV] {ALTV} J.~Adams, M.~van Leeuwen, P.~Trapa and D.~Vogan, {\it Unitary representations of real reductive groups}, Ast{\'erisque} \textbf{417} (2020).

\bibitem[BSi]{Baldoni}
M.V.~Baldoni-Silva, {\it The unitary dual of $Sp(n,1), n\ge 2$}, Duke Math. Journal, vol. 48, no. 3, (1981), 549--583.

\bibitem[BB]{BB}
M.W.~Baldoni-Silva, D.~Barbasch, {\it The unitary spectrum for real rank one groups}, Invent. Math. {\bf 72} (1983), 27--55.

\bibitem[B1]{B1}
D.~Barbasch,
{\it  The unitary dual for complex classical Lie groups},
Invent. Math.  \textbf{96} (1989), 103--176.

\bibitem[B2]{B2}
D.~Barbasch, {\it  The  unitary spherical spectrum for split classical groups},
Journal of Inst. of Math. Jussieu \textbf{9} (2010) Issue 2, pp.~265--356.

\bibitem[B3]{B3}
D.~Barbasch  {Unipotent representations and the dual pair
  correspondence},  Representation theory, number theory, and
invariant theory, Progr. Math., 323, Birkhauser/Springer,  (2017) pp.~47--85.

\bibitem[BP]{BP}
D.~Barbasch, P. Pand\v zi\'c,
{\it  Dirac cohomology and unipotent representations of complex groups},
Noncommutative geometry and global analysis,
Contemp. Math. \textbf{546}, Amer. Math. Soc., Providence, RI, 2011, pp.~1--22.

\bibitem[BV]{BV}
D.~Barbasch, D.~Vogan,
{\it  Unipotent representations of complex semisimple Lie groups},
Ann. of Math. \textbf{121} (1985), 41--110.

\bibitem[DD]{DD} J.~Ding, C.-P.~Dong,
{\it Unitary representations with Dirac cohomology: a finiteness result for complex Lie groups}, Forum Math. {\bf 32} (4) (2020), 941--964.

\bibitem[D1]{D1}  C.-P.~Dong,
{\it On the Dirac cohomology of complex Lie group representations},
Transform. Groups \textbf{18} (1) (2013), 61--79. Erratum: Transform.
Groups \textbf{18} (2) (2013), 595--597.

\bibitem[D2]{D2}   C.-P.~Dong, {\it Unitary representations with non-zero Dirac cohomology for complex $E_6$}, Forum Math. {\bf 31} (1) (2019), 69--82.

\bibitem[DW]{DW}   C.-P.~Dong, K.D.~Wong {\it Dirac series for complex $E_7$}, arXiv:2110.00694, accepted for publication in Forum Math.
%
\bibitem[E]{E}
  T.~Enright, {\it Relative Lie algebra cohomology and unitary representations of complex groups}, Duke Math. Journal, vol. 46, no. 3, (1979), pp. 513-525


\bibitem[HP1]{HP1}
J.-S. Huang, P. Pand\v{z}i\'{c},
{\it Dirac cohomology, unitary representations and a proof of a conjecture of
Vogan}, J. Amer. Math. Soc. \textbf{15} (2002), 185--202.

\bibitem[HP2]{HP2}
J.-S. Huang, P. Pand\v{z}i\'{c},
{\it Dirac Operators in Representation Theory},
Mathematics: Theory and Applications, Birkhauser, 2006.


\bibitem[Kn]{Kn}
A.~Knapp,
{\it  Representation Theory of Semisimple Groups},
Princeton University Press, Princeton NJ, 1986.


\bibitem[KnV]{KnV}
A.~Knapp, D.~Vogan,
{\it  Cohomological induction and unitary representations},
Princeton University Press, Princeton NJ, 1995.

\bibitem[L]{L}
G.~Lusztig,
{\it Characters of reductive groups over a finite field},
Annals of Mathematics Studies, Princeton University Press, 1984, no. 107.


%
\bibitem[LMM]{LMM} I.~Losev, L.~Mason-Brown, D.~Matvieievskyi, {\it
Unipotent ideals and Harish-Chandra bimodules}, preprint 2021,
arxiv:2108.03453.

\bibitem[P1]{P1} R.~Parthasarathy, {\it Dirac operators and the discrete
series}, Ann. of Math. \textbf{96} (1972), 1--30.

\bibitem[P2]{P2} R.~Parthasarathy, {\it Criteria for the unitarizability of some highest weight modules},
Proc. Indian Acad. Sci. \textbf{89} (1) (1980), 1--24.

\bibitem[PRV]{PRV} K.R.~Parthasarathy, R.~Ranga Rao, and
S.~Varadarajan, {\it Representations of complex semi-simple Lie
groups and Lie algebras}, Ann. of Math. \textbf{85} (1967),
383--429.

\bibitem[S]{S} S.~Salamanca-Riba, {\it On the unitary dual of some classical Lie
groups}, Compositio Math. \textbf{68} (1988), 251--303.

\bibitem[SV]{SV} S.~Salamanca-Riba, D.~Vogan, {\it On the classification of unitary representations of reductive Lie
groups}, Ann. of Math. \textbf{148} (3) (1998), 1067--1133.

\bibitem[Sa]{Sa}
  S.~Salamance-Riba, {\it On the unitary dual of real reductive groups and the $A_\fk q(\la)-$modules:The strongly regular case}, Duke Math. Journal, vol. 96, no. 3, (1999), 521-546.


\bibitem[V1]{V1}
D.~Vogan,
{\it  The unitary dual of $GL(n)$ over an archimedean field},
Invent. Math. \textbf{83} (1986), 449--505.

\bibitem[V2]{V2}
D.~Vogan,
{\it Dirac operator and unitary representations},
3 talks at MIT Lie groups seminar, Fall 1997.

\bibitem[VZ]{VZ}
D.~Vogan, G.~Zuckerman,
{\it Unitary Representations with non-zero cohomology},
Compositio Math.
\textbf{53} (1984), 51--90.


\bibitem[Zh]{Zh} D.~P.~Zhelobenko, {\it Harmonic analysis on complex semisimple Lie
groups}, Mir, Moscow, 1974.


\bibitem[At]{At} Atlas of Lie Groups and Representations, version 1.0.8, September 2020. See www.liegroups.org for more about the software.


\end{thebibliography}
\end{document}